\documentclass[3p,11pt]{elsarticle}
\usepackage{amssymb}
\usepackage{amsmath}
\usepackage{amsthm}

\DeclareMathAlphabet{\mathcal}{OMS}{cmsy}{m}{n}

\makeatletter
\def\ps@pprintTitle{%
 \let\@oddhead\@empty
 \let\@evenhead\@empty
 \def\@oddfoot{\centerline{\thepage}}%
 \let\@evenfoot\@oddfoot}
\makeatother

\newcommand{\bbC}{\mathbb{C}}
\newcommand{\bbF}{\mathbb{F}}
\newcommand{\bbH}{\mathbb{H}}
\newcommand{\bbR}{\mathbb{R}}
\newcommand{\bbT}{\mathbb{T}}
\newcommand{\bbZ}{\mathbb{Z}}

\newcommand{\bfA}{\mathbf{A}}
\newcommand{\bfB}{\mathbf{B}}
\newcommand{\bfG}{\mathbf{G}}
\newcommand{\bfI}{\mathbf{I}}
\newcommand{\bfJ}{\mathbf{J}}
\newcommand{\bfL}{\mathbf{L}}
\newcommand{\bfS}{\mathbf{S}}
\newcommand{\bfT}{\mathbf{T}}
\newcommand{\bfx}{\mathbf{x}}
\newcommand{\bfX}{\mathbf{X}}
\newcommand{\bfy}{\mathbf{y}}
\newcommand{\bfY}{\mathbf{Y}}
\newcommand{\bfz}{\mathbf{z}}
\newcommand{\bfZ}{\mathbf{Z}}

\newcommand{\bfzero}{\boldsymbol{0}}
\newcommand{\bfone}{\boldsymbol{1}}
\newcommand{\bfdelta}{\boldsymbol{\delta}}
\newcommand{\bfphi}{\boldsymbol{\varphi}}
\newcommand{\bfPhi}{\boldsymbol{\Phi}}
\newcommand{\bfPsi}{\boldsymbol{\Psi}}
\newcommand{\bfPi}{\boldsymbol{\Pi}}

\newcommand{\calB}{\mathcal{B}}
\newcommand{\calG}{\mathcal{G}}
\newcommand{\calO}{\mathcal{O}}
\newcommand{\calV}{\mathcal{V}}

\newcommand{\rmc}{\mathrm{c}}
\newcommand{\rmK}{\mathrm{K}}
\newcommand{\rmi}{\mathrm{i}}
\newcommand{\rmT}{\mathrm{T}}

\newcommand{\orb}{\operatorname{orb}}
\newcommand{\id}{\operatorname{id}}
\newcommand{\Tr}{\operatorname{Tr}}
\newcommand{\GQ}{\operatorname{GQ}}
\newcommand{\SRG}{\operatorname{SRG}}
\newcommand{\Span}{\operatorname{span}}
\newcommand{\BIBD}{\operatorname{BIBD}}

\newcommand{\abs}[1]{|{#1}|}
\newcommand{\set}[1]{\{{#1}\}}
\newcommand{\bigset}[1]{\bigl\{{#1}\bigr\}}
\newcommand{\norm}[1]{\|{#1}\|}
\newcommand{\gen}[1]{\langle{#1}\rangle}
\newcommand{\ip}[2]{\langle{#1},{#2}\rangle}

\newcommand{\alphi}{\renewcommand{\labelenumi}{(\alph{enumi})}}
\newcommand{\romani}{\renewcommand{\labelenumi}{(\roman{enumi})}}

\newtheorem{theorem}{Theorem}[section]
\newtheorem{lemma}[theorem]{Lemma}
\theoremstyle{definition}
\newtheorem{definition}[theorem]{Definition}
\newtheorem{example}[theorem]{Example}

\begin{document}
\begin{frontmatter}
\title{Polyphase equiangular tight frames and abelian generalized quadrangles}

\author[AFIT]{Matthew Fickus}
\ead{Matthew.Fickus@gmail.com}
\author[Cincy]{John Jasper}
\author[AFIT]{Dustin G.\ Mixon}
\author[AFIT]{Jesse D.\ Peterson}
\author[AFIT]{Cody E.\ Watson}

\address[AFIT]{Department of Mathematics and Statistics, Air Force Institute of Technology, Wright-Patterson AFB, OH 45433}
\address[Cincy]{Department of Mathematical Sciences, University of Cincinnati, Cincinnati, OH 45221}

\begin{abstract}
An equiangular tight frame (ETF) is a type of optimal packing of lines in a finite-dimensional Hilbert space.
ETFs arise in various applications, such as waveform design for wireless communication, compressed sensing, quantum information theory and algebraic coding theory.
In a recent paper, signature matrices of ETFs were constructed from abelian distance regular covers of complete graphs.
We extend this work, constructing ETF synthesis operators from abelian generalized quadrangles, and vice versa.
This produces a new infinite family of complex ETFs as well as a new proof of the existence of certain generalized quadrangles.
This work involves designing matrices whose entries are polynomials over a finite abelian group.
As such, it is related to the concept of a polyphase matrix of a finite filter bank.
\end{abstract}

\end{frontmatter}

\section{Introduction}

An equiangular tight frame is a type of optimal packing of $n$ lines in a $d$-dimensional (real or complex) Hilbert space $\bbH_d$, where $n\geq d$.
To be precise, Welch~\cite{Welch74} gives the following lower bound on the \textit{coherence} of any sequence $\set{\bfphi_i}_{i=1}^n$ of $n$ nonzero equal-norm vectors in $\bbH_d$:
\begin{equation}
\label{equation.Welch bound}
\max_{i\neq j}\frac{\abs{\ip{\bfphi_i}{\bfphi_j}}}{\norm{\bfphi_i}\norm{\bfphi_j}}\geq\bigg[\frac{n-d}{d(n-1)}\biggr]^{\frac12}.
\end{equation}
It is well-known~\cite{StrohmerH03} that nonzero equal-norm vectors $\set{\bfphi_i}_{i=1}^n$ achieve equality in this \textit{Welch bound} if and only if they form an \textit{equiangular tight frame} (ETF) for $\bbH_d$, namely when there exist constants $w$ and $a>0$ such that
\begin{equation*}
\abs{\ip{\bfphi_i}{\bfphi_j}}
=w,\quad \forall i,j=1,\dotsc,n,\ i\neq j,
\qquad
a\bfx=\sum_{i=1}^{n}\ip{\bfphi_i}{\bfx}\bfphi_i,\quad \forall \bfx\in\bbH_d.
\end{equation*}

Having minimal coherence, ETFs arise in a number of applications, including waveform design for wireless communication~\cite{StrohmerH03}, compressed sensing~\cite{BajwaCM12,BandeiraFMW13}, quantum information theory~\cite{RenesBSC04,Zauner99} and algebraic coding theory~\cite{JasperMF14}.
They also seem to be rare, and only a few methods for constructing infinite families of them are known.
With the exception of orthonormal bases and regular simplices,
each of these methods involve some type of combinatorial design.

Real ETFs in particular are equivalent to a subclass of \textit{strongly regular graphs} (SRGs)~\cite{HolmesP04,Waldron09},
and this subject has a rich literature~\cite{Brouwer07,Brouwer15,CorneilM91}.
See~\cite{SustikTDH07} for necessary conditions on the $d$ and $n$ parameters of real ETFs.
Conference matrices, Hadamard matrices, Paley tournaments, and quadratic residues are interrelated,
and they lead to infinite families of ETFs whose \textit{redundancy} $\frac nd$ is either nearly or exactly $2$~\cite{StrohmerH03,HolmesP04,Renes07,Strohmer08}.
Other constructions allow $d$ and $n$ to be chosen independently, and almost arbitrarily, up to an order of magnitude.
For example, this flexibility is offered by the \textit{harmonic ETFs} of~\cite{StrohmerH03,XiaZG05,DingF07}, which are obtained by restricting the Fourier basis on a finite abelian group to a difference set for that group.
Another flexible construction is the \textit{Steiner ETFs} of~\cite{FickusMT12}, which arise from a tensor-like product of a simplex and the incidence matrix of a certain type of finite geometry.

To be clear, many of these ideas are rediscoveries or reimaginings of more classical results.
In particular, see~\cite{vanLintS66,Seidel76} for early studies of the relationship between real ETFs and SRGs, and see~\cite{Rankin56}, \cite{Turyn65} and \cite{GoethalsS70} for precursors of the Welch bound, harmonic ETFs and Steiner ETFs, respectively.
That said, much recent progress has been made.
For example, new infinite families of ETFs are given in~\cite{FickusJMP16,FickusMJ16}.
Some of the ETFs in~\cite{FickusJMP16} are real, and these have led to new SRGs.
Other new SRGs have recently arisen from a new relation~\cite{Yu14} between certain real ETFs and SRGs~\cite{FickusMJPW15}.
Until recently, the existence of real ETFs with $(d,n)$ parameters $(19,76)$ and $(20,96)$ were longstanding open problems.
Both have now been ruled out with computer-assisted arguments~\cite{AzarijaM15,AzarijaM16,Yu15}.
Notably,~\cite{FickusMJ16} shows how to construct complex ETFs of these sizes.

In this paper, we extend the results of yet another recent development in the field,
namely the construction of an ETF from an abelian \textit{distance regular cover of a complete graph} (DRACKN)~\cite{CoutinkhoGSZ16}.
Though DRACKNs have been studied for a long time~\cite{GodsilH92},
only a few explicit methods for constructing infinite families of them are known.
Moreover, only some of these families are known to possess the abelian structure that yields ETFs via the method of~\cite{CoutinkhoGSZ16}.
Also, as detailed in the next section, it is unclear whether any of the ETFs produced in~\cite{CoutinkhoGSZ16} are new: their $(d,n)$ parameters match those of other known ETFs, meaning they might arise by rearranging and/or phasing the frame vectors of those ETFs.
Nevertheless, this construction is remarkable since it has a partial converse:
for any prime $p$, $(n,p,c)$-DRACKNs are equivalent to ETFs whose Gram matrices contain only $p$th roots of unity~\cite{CoutinkhoGSZ16}.
In particular, for any prime $p$, a construction of \cite{BodmannPT09,BodmannE10} yields ETFs that are equivalent to $(p^2,p,p)$-DRACKNs.
More recently, \cite{FickusJMP16} used this equivalence to construct a new family of DRACKNs from a new family of ETFs.

In this paper, we discuss the connections between~\cite{CoutinkhoGSZ16} and certain incidence structures known as \textit{generalized quadrangles} (GQs).
See~\cite{Payne07} for an overview of GQs, and~\cite{PayneT09} for details.
It has long been known that every GQ produces a DRACKN~\cite{Brouwer84,GodsilH92}.
We show that a certain type of GQ produces an abelian DRACKN and thus an ETF.
As we shall see, each of these so-called \textit{abelian GQs} will be a cover of a \textit{balanced incomplete block design} (BIBD).
Paralleling~\cite{CoutinkhoGSZ16}, we show that under certain conditions, the existence of abelian GQs is equivalent to existence of certain types of ETFs.
Much of this work involves matrices whose entries lie in the group ring $\bbC\calG$ of some finite abelian group $\calG$.
We choose to follow the conventions of the wavelet literature and regard such matrices as \textit{polyphase matrices} whose entries are polynomials over $\calG$.
Such matrices are a standard tool for the analysis of filter banks~\cite{Vaidyanathan92},
in particular their frame properties~\cite{CvetkovicV98,BolcskeiHF98}.

In the next section, we present the background material on ETFs, polyphase matrices and DRACKNs that we will need in order to prove our main results.
In Section~3, we show how the incidence matrix of a BIBD can sometimes be phased so as to produce a matrix whose columns form an ETF for their span.
In Theorems~\ref{theorem.phased BIBD ETFs} and~\ref{theorem.phased BIBD necessary conditions} we provide various characterizations and necessary conditions on such phased BIBD ETFs.
This construction is novel:
unlike other known constructions of ETFs,
the vectors span a space of much smaller dimension than the ambient space.
In particular, in contrast to harmonic and Steiner ETFs,
it is easy to see that the vectors in a phased BIBD ETF are equiangular,
but difficult to see that they form a tight frame for their span.

From Section~4 onward, we focus on polyphase BIBD ETFs, which are polynomial-valued versions of phased BIBD ETFs.
As shown in Theorems~\ref{theorem.polyphase BIBD ETF} and~\ref{theorem.abelian GQ}, polyphase BIBD ETFs yield abelian DRACKNs and in a special case correspond to abelian GQs.
In Theorem~\ref{theorem.ETF implies GQ},
we further show that under certain conditions, the existence of certain types of ETFs implies the existence of abelian GQs.
We conclude in Section~5 with several constructions of abelian GQs.
In particular, for any prime power $q$, Theorem~\ref{theorem.(q^2,q,q) polyphase BIBD ETFs} gives a construction of an abelian $\GQ(q-1,q+1)$ while Theorem~\ref{theorem.(q^3+1,q+1,q+1) polyphase BIBD ETFs} shows that a construction attributed to Brouwer in~\cite{Godsil92} yields an abelian $\GQ(q,q^2)$.
This latter construction yields ETFs with parameters
\begin{equation}
\label{equation.new ETF parameters}
d=q(q^2-q+1),
\quad
n=q^3+1,
\quad
n-d=q^2-q+1.
\end{equation}
As detailed in Section~5,
when $q$ is odd, these ETFs are real and correspond to a known infinite family of SRGs constructed in~\cite{Godsil92};
when $q$ is even, these ETFs are complex and are verifiably new whenever $q-1$ is not an odd prime power, which occurs infinitely often.
We note that the redundancy $\frac{n}{n-d}$ of these ETFs is unbounded,
a property only known to be shared by some special families of harmonic ETFs~\cite{XiaZG05,DingF07} and Steiner ETFs~\cite{FickusMT12}, as well as by those ETFs arising from hyperovals in finite projective planes~\cite{FickusMJ16}.

\section{Background}

This paper leverages insights from three different areas of research, namely equiangular tight frames,
polyphase representations of filter banks, and distance regular graphs.
In this section, we discuss concepts and notation from each of these areas that we use later on.
The reader who is already familiar with these areas may advance to Section~3, where we begin to introduce new ideas.

\subsection{Equiangular tight frames and signature matrices}

Throughout, let $d\leq n$ be positive integers, let the field of scalars $\bbF$ be either $\bbR$ or $\bbC$, and let $\bbH_d$ be a $d$-dimensional subspace of $\bbF^m$.
The \textit{synthesis operator} of a finite sequence of vectors $\set{\bfphi_i}_{i=1}^{n}$ in $\bbH_d$ is $\bfPhi:\bbF^n\rightarrow\bbH_d$, $\bfPhi\bfy:=\sum_{i=1}^{n}\bfy(i)\bfphi_i$.
Since $\bbH_d$ is a subspace of $\bbF^m$, this operator can be regarded as the $m\times n$ matrix whose $i$th column is $\bfphi_i$.
The adjoint of $\bfPhi$ is the $n\times m$ \textit{analysis operator} $\bfPhi^*:\bbF^m\rightarrow\bbF^n$, $(\bfPhi^*\bfx)(n)=\ip{\bfphi_n}{\bfx}$,
where the inner product on $\bbF^m$ is taken to be the standard dot product, chosen to be conjugate-linear in its first argument.
Composing these two operators gives the $m\times m$ \textit{frame operator} $\bfPhi\bfPhi^*:\bbH_d\rightarrow\bbH_d$,
$\bfPhi\bfPhi^*\bfx=\sum_{i=1}^{n}\ip{\bfphi_i}{\bfx}\bfphi_i$ as well as the $n\times n$ \textit{Gram matrix} $\bfPhi^*\bfPhi$ whose $(i,j)$th entry is $\ip{\bfphi_i}{\bfphi_j}$.

Such a sequence of vectors $\set{\bfphi_i}_{i=1}^{n}$ is a \textit{tight frame} for $\bbH_d$ if there exists $a>0$ such that
$\bfPhi\bfPhi^*\bfx=a\bfx$ for all $\bfx\in\bbH_d$.
When $a$ is specified, it is an \textit{$a$-tight frame}.
It is \textit{equal norm} if there exists $r>0$ such that $\norm{\bfphi_i}^2=r$ for all $i$,
and is \textit{equiangular} if it is equal norm and there exists $w$ such that $\abs{\ip{\bfphi_i}{\bfphi_j}}=w$ for all $i\neq j$.
When $\set{\bfphi_i}_{i=1}^{n}$ is both equiangular and a tight frame, it is an \textit{equiangular tight frame} (ETF).

As mentioned in the introduction, the ETFs we construct in this paper are unusual in that their synthesis operators $\bfPhi$ are most naturally represented as $m\times n$ matrices that are ``tall, skinny" and rank-deficient, that is, $d<n<m$.
The equiangularity of these vectors will follow immediately from the particular manner in which they are constructed.
However, it will be far from obvious that they form a tight frame for their span.
Here, our main tool will be the following result from~\cite{FickusMJ16}:

\begin{lemma}[Lemma~1 of~\cite{FickusMJ16}]
\label{lemma.tight for subspace}
For any vectors $\set{\bfphi_i}_{i=1}^{n}$ in $\bbF^m$, let $\bfPhi$ be the $m\times n$ matrix whose $i$th column is $\bfphi_i$ for all $i$.
For any $a>0$,
the following are equivalent:
\begin{enumerate}
\romani
\item
$\set{\bfphi_i}_{i=1}^{n}$ forms an $a$-tight frame for its span,
\item
$\bfPhi\bfPhi^*\bfPhi=a\bfPhi$,
\item
$(\bfPhi\bfPhi^*)^2=a\bfPhi\bfPhi^*$,
\item
$(\bfPhi^*\bfPhi)^2=a\bfPhi^*\bfPhi$.
\end{enumerate}

Also, if $\set{\bfphi_i}_{i=1}^{n}$ is contained in a $d$-dimensional subspace $\bbH_d$ of $\bbF^m$, then it forms an $a$-tight frame for $\bbH_d$ if and only if  $\bfPhi\bfPhi^*=a\bfPi$ where $\bfPi$ is the $m\times m$ orthogonal projection matrix onto $\bbH_d$.
As such, $\set{\bfphi_i}_{i=1}^{n}$ forms an ETF for its span if and only if (ii)--(iv) hold and $\set{\bfphi_i}_{i=1}^{n}$ is equiangular.
In this case, letting $r=\norm{\bfphi_i}^2$, the dimension $d$ can be computed from either the tight frame constant or the equiangularity constant:
\begin{equation}
\label{equation.ETF parameter relations}
a=\frac{rn}d,
\quad
w=r\biggl[\frac{n-d}{d(n-1)}\biggr]^{\frac12}.
\end{equation}

Alternatively, equal norm vectors $\set{\bfphi_i}_{i=1}^{n}$ in a subspace $\bbH_d$ of $\bbF^m$ of dimension $d$ form an ETF for $\bbH_d$ if and only if they achieve equality in~\eqref{equation.Welch bound}.
\end{lemma}

One consequence of this lemma is that if the columns of $\bfPhi$ form an ETF for their span, then so do the columns of $\bfPhi^*\bfPhi$.
Indeed, it is well known that in this case the columns of $\bfPhi^*\bfPhi$ are simply an alternative representation of the columns of $\bfPhi$.
Conversely, if $\bfG$ is any self-adjoint matrix whose diagonal entries are some constant $r$, whose off-diagonal entries have constant modulus $w$, and which satisfies $\bfG^2=a\bfG$ for some $a>0$, then the columns of $\bfG$ form an ETF for their span.
This fact is also well known, and can be proven by taking the singular value decomposition of $\bfG$, or alternatively, taking $\bfPhi$ to be $\bfG$ in Lemma~\ref{lemma.tight for subspace}.

Any ETF has an associated \textit{signature matrix}, that is, a self-adjoint matrix $\bfS$ whose diagonal entries are zero and whose off-diagonal entries are unimodular.
In particular, if $\bfG$ is the Gram matrix of an ETF then $\bfS:=\frac1w(\bfG-r\bfI)$ is a signature matrix which (since $\bfG^2=a\bfG$) satisfies a quadratic equation:
\begin{equation}
\label{equation.signature matrix equation}
\bfS^2
=\frac{a-2r}{w}\bfS+\frac{r(a-r)}{w^2}\bfI
=\frac{(\tfrac nd-2)\sqrt{n-1}}{\sqrt{\tfrac nd-1}}\bfS+(n-1)\bfI.
\end{equation}
Here, the coefficient $\delta=\frac{a-2r}{w}$ of $\bfS$ determines the dimension $d$ of the span of the ETF's vectors:
\begin{equation}
\label{equation.dimension from signature}
\delta=\frac{(\tfrac nd-2)\sqrt{n-1}}{\sqrt{\tfrac nd-1}}
\qquad\Longleftrightarrow\qquad
d
=\frac n2\biggl[1-\frac{\delta}{\sqrt{\delta^2+4(n-1)}}\biggr].
\end{equation}
Conversely, if $\bfS$ is any signature matrix that satisfies $\bfS^2=\delta\bfS+(n-1)\bfI$ for some $\delta$,
then defining $d$ according to~\eqref{equation.dimension from signature} and then defining $a$ and $w$ from~\eqref{equation.ETF parameter relations} for an arbitrary choice of $r>0$,
the matrix $\bfG:=r\bfI+w\bfS$ satisfies $\bfG^2=a\bfG$ and so is the Gram matrix of an $n$-vector ETF for some $d$-dimensional Hilbert space.

Thus, ETFs are equivalent to signature matrices which satisfy quadratic equations.
One immediate consequence of this equivalence is that if there exists an $n$-vector ETF for a space of dimension $d$ then there also exists an $n$-vector ETF for a space of dimension $n-d$: if $\bfS^2=\delta\bfS+(n-1)\bfI$ then $(-\bfS)^2=-\delta(-\bfS)+(n-1)\bfI$, and the $d$ parameters~\eqref{equation.dimension from signature} arising from $\delta$ and $-\delta$ clearly sum to $n$.
Two ETFs whose signature matrices are the negatives of each other are called \textit{Naimark complements}.

\subsection{Convolution algebras over finite abelian groups and polyphase matrices}

For any finite abelian group $\calG$, the \textit{convolution algebra} over $\calG$ is the group ring $\bbC\calG$,
namely the vector space $\bbC^\calG:=\set{\bfx:\calG\rightarrow\bbC}$ of all $\calG$-indexed complex vectors equipped with the \textit{convolution} product
\smash{$(\bfx_1*\bfx_2)(g):=\sum_{g'\in\calG}\bfx_1(g')\bfx_2(g-g')$}.
This product naturally arises from regarding each $\bfx\in\bbC\calG$ as a formal polynomial over $\calG$,
namely as its \textit{$z$-transform} $\bfx(z):=\sum_{g\in\calG}\bfx(g)z^g$:
\begin{equation*}
\sum_{g\in\calG}\bfx_1(g)z^g\sum_{g\in\calG}\bfx_2(g)z^g
:=\sum_{g\in\calG}(\bfx_1*\bfx_2)(g)z^g
=\sum_{g_1\in\calG}\sum_{g_2\in\calG}\bfx_1(g_1)\bfx_2(g_2)z^{g_1+g_2}.
\end{equation*}
Under this product, $\bbC\calG$ is a commutative ring with multiplicative identity $\bfdelta_0$, where $\set{\bfdelta_g}_{g\in\calG}$ denotes the standard basis in $\bbC^{\calG}$.
This multiplicative identity has $z$-transform $1:=z^0=\bfdelta_0(z)$.

To more explicitly justify this polynomial notation,
note that since $\calG$ is a finite abelian group,
there exists a group isomorphism of the form \smash{$\eta:\calG\rightarrow\oplus_{i=1}^{j}\bbZ_{q_i}$} for some positive integers $\set{q_i}_{i=1}^{j}$.
This then implies \smash{$\sum_{g\in\calG}\bfx(g)z^g\mapsto\sum_{g\in\calG}\bfx(g)\prod_{i=1}^{j}z_i^{\eta(g)_i}$}
is an isomorphism from $\bbC\calG$ into the algebra of complex polynomials in $j$ variables $\bbC[z_1,\dotsc,z_j]$ modulo the ideal generated by \smash{$\set{z^{q_i}-1}_{i=1}^{j}$}.
That is, we can identify $z^g$ with $(z_1,\dotsc,z_j)^{(g_1,\dotsc,g_j)}:=(z_1^{g_1},\dotsc,z_j^{g_j})$
where $z_i^{q_i}\equiv 1$ for each $i=1,\dotsc,j$.

For any $\bfx\in\bbC\calG$, the corresponding \textit{filter} is the linear operator $\bfL:\bbC\calG\rightarrow\bbC\calG$, $\bfL\bfy=\bfx*\bfy$.
That is, $\bfL=\sum_{g\in\calG}\bfx(g)\bfT^g$,
where for each $g\in\calG$, $\bfT^g$ is the \textit{translation} operator
$\bfT^g:\bbC\calG\rightarrow\bbC\calG$, $(\bfT^g\bfy)(g'):=\bfy(g-g')$.
Note that writing elements of $\bbC\calG$ as $z$-transforms, we see the mapping
\begin{equation}
\label{equation.definition of x(T)}
\bfx(z)=\sum_{g\in\calG}\bfx(g)z^g
\mapsto\bfx(\bfT):=\sum_{g\in\calG}\bfx(g)\bfT^g,
\end{equation}
is an isomorphism from $\bbC\calG$ onto the algebra of all such filters.
In particular, identifying $\calG$ with $\oplus_{i=1}^{j}\bbZ_{q_i}$,
this mapping sends $z^g=(z_1^{g_1},\dotsc,z_j^{g_j})$ to $\bfT^g =\bfT_1^{g_1}\otimes\dotsb\otimes\bfT_j^{g_j}$ where, for each $i=1,\dotsc,j$, $\bfT_i$ is a circulant permutation matrix of order $q_i$.
Since translation operators are unitary, the adjoint of $\bfx(\bfT)$ is $[\bfx(\bfT)]^*=\tilde{\bfx}(\bfT)$ where $\tilde{\bfx}(g):=\overline{\bfx(-g)}$ is the \textit{involution} of $\bfx$.
In light of~\eqref{equation.definition of x(T)}, this suggests we define the \textit{conjugate} of any $z$-transform $\bfx(z)$ as $\overline{\bfx(z)}:=\tilde{\bfx}(z)$.

We often identify a linear operator $\bfL:\bbC^{\calG}\rightarrow\bbC^{\calG}$ with the matrix in $\bbC^{\calG\times\calG}:=\set{\bfL:\calG\times\calG\rightarrow\bbC}$ that represents it with respect to the standard basis.
That is, we let $\bfL(g',g''):=\ip{\bfdelta_{g'}}{\bfL\bfdelta_{g''}}$ for all $g',g''\in\calG$, where this inner product is the complex dot product over $\bbC^\calG$, taken to be conjugate-linear in its first argument.
Under this identification, each translation operator is a permutation matrix,
and a matrix $\bfL\in\bbC^{\calG\times\calG}$ is a filter if and only if it is \textit{circulant}, namely when $\bfL(g',g'')=\bfL(g'-g'',0)$ for all $g',g''\in\calG$.

A \textit{character} of $\calG$ is a homomorphism $\gamma:\calG\rightarrow\bbT:=\set{z\in\bbC: \abs{z}=1}$.
Any given character naturally extends to the linear functional $\gamma:\bbC\calG\rightarrow\bbC$,
$\gamma(\bfx):=\sum_{g\in\calG}\bfx(g)\gamma(g)$.
Each extended character $\gamma$ satisfies $\gamma(\bfx_1*\bfx_2)=\gamma(\bfx_1)\gamma(\bfx_2)$ for all $\bfx_1,\bfx_2\in\bbC\calG$ and so is a ring homomorphism from $\bbC\calG$ to $\bbC$.
Regarding each $\bfx\in\bbC\calG$ as a polynomial $\bfx(z)$,
we have $\gamma(\bfx_1(z)\bfx_2(z))=\gamma(\bfx_1(z))\gamma(\bfx_2(z))$.
To improve this clunky notation, let $\bfx(\gamma):=\gamma(\bfx)$ be the ``evaluation" of $\bfx(z)$ at $\gamma$.
Under this notation, the homomorphism from $\bbC\calG$ to $\bbC$ given by $\gamma$ becomes
\begin{equation}
\label{equation.definition of x(gamma)}
\bfx(z)=\sum_{g\in\calG}\bfx(g)z^g
\mapsto\bfx(\gamma):=\sum_{g\in\calG}\bfx(g)\gamma(g).
\end{equation}
Writing $\calG\cong\oplus_{i=1}^{j}\bbZ_{q_i}$ and regarding $\bfx(z)$ as a member of $\bbC[z_1,\dotsc,z_j]/\gen{\set{z^{q_i}-1}_{i=1}^{j}}$,
$\bfx(\gamma)$ indeed corresponds to evaluating $\bfx(z)$ at $(z_1,\dotsc,z_j)$ where each $z_i$ is a $q_i$th root of unity obtained by applying $\gamma$ to some fixed generator of $\bbZ_{q_i}$, regarded as a subgroup of $\calG$.
This notation is also consistent with conjugation:
\smash{$\overline{\gamma(g)}=\gamma(-g)$} for all $g\in\calG$, and so the complex conjugate of the scalar $\bfx(\gamma)$ is \smash{$\overline{\bfx(\gamma)}=\tilde{\bfx}(\gamma)$},
namely the evaluation of the conjugate polynomial $\overline{\bfx(z)}=\tilde{\bfx}(z)$ at $\gamma$.

The set $\Gamma$ of all characters of $\calG$ is itself an abelian group under pointwise multiplication.
In fact, writing $\calG\cong\oplus_{i=1}^{j}\bbZ_{q_i}$ and identifying each $\gamma\in\Gamma$ with a unique $j$-tuple of $q_i$th roots of unity, we see that $\Gamma\cong\calG$.
Regarded as vectors in $\bbC^\calG$, the characters form an orthogonal basis for $\bbC^\calG$.
Since $\ip{\gamma}{\bfx}=\bfx(\gamma^{-1})$, this implies that $\bfx_1(z)=\bfx_2(z)$ if and only if $\bfx_1(\gamma)=\bfx_2(\gamma)$ for all $\gamma$.
The mapping \smash{$\bfx\mapsto\set{\bfx(\gamma^{-1})}_{\gamma\in\Gamma}$} is the \textit{discrete Fourier transform} over $\calG$.

Much of our work here involves matrices $\bfPhi$ whose entries lie in $\bbC\calG$.
When regarding these entries as polynomials, we denote such a matrix as $\bfPhi(z)$ and, following the wavelets literature, call it a \textit{polyphase} matrix.
Applying~\eqref{equation.definition of x(T)} to each entry of $\bfPhi(z)$ yields a block matrix with circulant blocks, namely the corresponding \textit{filter bank} $\bfPhi(\bfT)$.
Meanwhile, applying~\eqref{equation.definition of x(gamma)} to each entry of $\bfPhi(z)$ yields the scalar matrix $\bfPhi(\gamma)$.

Two polyphase matrices $\bfPhi(z)$ and $\bfPsi(z)$ can be summed or multiplied in the usual way, provided they are of the appropriate size and their entries lie in the same group ring $\bbC\calG$.
Since~\eqref{equation.definition of x(T)} is a ring isomorphism, these correspond to usual sums and products of $\bfPhi(\bfT)$ and $\bfPsi(\bfT)$ as block matrices.
Similarly, the fact that~\eqref{equation.definition of x(gamma)} is a ring homomorphism implies that evaluating a sum or product of $\bfPhi(z)$ and $\bfPsi(z)$ at $\gamma$ is equivalent to just summing or multiplying the scalar matrices $\bfPhi(\gamma)$ and $\bfPsi(\gamma)$.
Moreover, since~\eqref{equation.definition of x(T)} and~\eqref{equation.definition of x(gamma)} preserve conjugation,
$[\bfPhi(\bfT)]^*=\bfPhi^*(\bfT)$ and $[\bfPhi(\gamma)]^*=\bfPhi^*(\gamma)$ provided the $(j,i)$th entry of $\bfPhi^*(z)$ is defined as the conjugate (involution) of the $(i,j)$th entry of $\bfPhi(z)$.
Further note that since the characters form an orthonormal basis for $\bbC^\calG$,
$\bfPhi(z)=\bfPsi(z)$ if and only if $\bfPhi(\gamma)=\bfPsi(\gamma)$ for all $\gamma\in\Gamma$.

These ideas give elegant proofs for some of the fundamental results of discrete wavelet transforms.
For example, a square filter bank matrix $\bfPhi(\bfT)$ is unitary if and only if $\bfPhi(\bfT)[\bfPhi(\bfT)]^*=\bfI$,
namely when its polyphase matrix satisfies $\bfPhi(z)\bfPhi^*(z)=\bfI$,
namely when $\bfPhi(\gamma)[\bfPhi(\gamma)]^*=\bfI$ for all $\gamma$.
The advantage here is that $\bfPhi(\gamma)$ is substantially smaller than $\bfPhi(\bfT)$, and thus easier to design and analyze.
In traditional wavelets, for example, $\bfPhi(\bfT)$ is a $2\times 2$ array of circulant blocks,
meaning $\bfPhi(\gamma)$ is just a $2\times 2$ scalar matrix for each $\gamma\in\Gamma$.
As we now discuss, a similar application of these ideas was recently used to construct ETFs.

\subsection{Strongly regular graphs and abelian distance regular antipodal covers of complete graphs}

Since ETF signature matrices satisfy \eqref{equation.signature matrix equation},
it is not surprising that they are related to graphs whose adjacency matrices satisfy certain quadratic-like equations.
For instance, an $\SRG(v,k,\lambda,\mu)$ is a graph with $v$ vertices, each having $k$ neighbors, with the property that any two neighbors have $\lambda$ neighbors in common, while any two nonadjacent vertices have $\mu$ neighbors in common.
That is, a graph is strongly regular if and only if its adjacency matrix $\bfA$ satisfies
\begin{equation}
\label{equation.definition of SRG}
\bfA^2=(\lambda-\mu)\bfA+(k-\mu)\bfI+\mu\bfJ,
\quad
\text{i.e.}
\quad
\bfA^2(i,j)
=\left\{\begin{array}{cl}k,&i=j,\\\lambda,&i\neq j,\ \bfA(i,j)=1,\\\mu,&i\neq j,\ \bfA(i,j)=0,\end{array}\right.
\end{equation}
where $\bfJ$ is an all-ones matrix.
An SRG's parameters are interrelated: conjugating~\eqref{equation.definition of SRG} by an all-ones vector $\bfone$ gives $k(k-\lambda-1)=(v-k-1)\mu$.
It is well-known that real $n$-vector ETFs are equivalent to SRGs on $v=n-1$ vertices where $k=2\mu$;
see~\cite{Waldron09,FickusW15} for details.
Recently, it was also shown that real ETFs whose Gram matrices have $\bfone$ as an eigenvector are equivalent to SRGs on $v=n$ vertices where $v=4k-2\lambda-2\mu$~\cite{FickusMJPW15}.
Every SRG is \textit{distance regular}, meaning that the number of vertices at given distances from vertex $i$ and $j$ depends only on those distances and the distance between $i$ and $j$.

Recently, a second type of distance regular graph was used to construct ETFs~\cite{CoutinkhoGSZ16},
namely a \textit{distance regular antipodal cover of a complete graph} (DRACKN).
To elaborate, let $\rmK_n$ denote the complete graph on $n$ vertices.
For a given integer $f\geq 2$, an \textit{$f$-fold cover of $\rmK_n$} is a graph whose vertex set can be partitioned into $n$ subsets of size $f$, dubbed \textit{fibers}, so that no two vertices in the same fiber are adjacent and so that there exists a perfect matching between any two distinct fibers~\cite{CoutinkhoGSZ16}.
That is, a graph is an $f$-fold cover of $\rmK_n$ if it has an $nf\times nf$ adjacency matrix $\bfA$ which is an $n\times n$ array of blocks of size $f\times f$ whose $(i,j)$th block $\bfA(i,j)$ is zero when $i=j$ and is a permutation matrix when $i\neq j$.

Now suppose that a given $f$-fold cover of $\rmK_n$ has the property that any two nonadjacent vertices from distinct fibers have exactly $c$ neighbors in common, where $c$ is some positive constant.
By Lemma~3.1 of~\cite{GodsilH92}, any such graph has diameter $3$ and is \textit{antipodal}:
fibers are equivalence classes, provided two vertices are equated when the distance between them is three (or zero).
This same lemma shows that any such graph is distance regular.
In fact, its proof shows that any two adjacent vertices have $(n-1)-1-(f-1)c$ neighbors in common.
Altogether, these facts imply the number of two-step paths from the $x$th vertex of the $i$th fiber to the $y$th vertex of the $j$th fiber is
\begin{equation*}
\bfA^2(i,j;x,y)
=\left\{\begin{array}{cl}
n-1,            &i=j, x=y,\\
0,              &i=j, x\neq y,\\
c,              &i\neq j, \bfA(i,j;x,y)=0,\\
(n-1)-1-(f-1)c, &i\neq j, \bfA(i,j;x,y)=1,\end{array}\right.
\end{equation*}
that is,
\begin{equation}
\label{equation.DRACKN quadratic condition}
\bfA^2=(n-fc-2)\bfA+(n-1)\bfI_{n}\otimes\bfI_f+c(\bfJ_n-\bfI_n)\otimes\bfJ_f.
\end{equation}
The first component of these tensor products refers to the base and the second one to the fiber.

For positive integers $n$, $f$ and $c$ with $f\geq2$,
a graph is said to be an $(n,f,c)$-DRACKN if it has an $(n\times n)\times(f\times f)$ block adjacency matrix $\bfA$ whose diagonal blocks are zero, whose off-diagonal blocks are permutations, and which satisfies~\eqref{equation.DRACKN quadratic condition}.
See~\cite{GodsilH92} for a thorough discussion of various types of DRACKNs.
Our work here focuses on $(n,f,c)$-DRACKNs that happen to be \textit{abelian},
namely when the off-diagonal blocks of $\bfA$ all arise as translation operators over some abelian group $\calG$ of order $f$.
As discussed in Section~7 of~\cite{GodsilH92},
this means each off-diagonal block of $\bfA$ may be regarded as a standard basis element of the group ring $\bbC\calG$.
Under this identification, quantities $\bfI_n$ and $\bfJ_n$ that appear in~\eqref{equation.DRACKN quadratic condition} become $\bfdelta_0$ and $\sum_{g\in\calG}\bfdelta_g$, respectively.
Writing these as formal polynomials of $\calG$ leads to our following restatement of their definition:

\begin{definition}
\label{definition.abelian DRACKN}
For positive integers $n$, $f$ and $c$ with $f\geq2$ and an abelian group $\calG$ of order $f$,
an \textit{abelian $(n,f,c)$-DRACKN} is a self-adjoint $n\times n$ polyphase matrix $\bfA(z)$ whose diagonal entries are zero, whose off-diagonal entries are monomials ($z$-transforms of standard basis elements) over $\calG$, and which satisfies
\begin{equation}
\label{equation.DRACKN property}
[\bfA(z)]^2=(n-fc-2)\bfA(z)+(n-1)\bfI+c\bfone(z)(\bfJ-\bfI),
\end{equation}
where $\bfone(z):=\sum_{g\in\calG}z^g$ is the geometric sum over $\calG$.
\end{definition}

To be clear, the entries of all matrices here lie in the ring $\bbC\calG$.
In particular, the diagonal entries of $\bfI$ and the off-diagonal entries of $\bfJ-\bfI$ are the multiplicative identity in $\bbC\calG$, namely $1:=\bfdelta_0(z)$.

As indicated by Theorems~3.2 and~4.1 of~\cite{CoutinkhoGSZ16}, abelian DRACKNs are intimately related to ETFs.
To elaborate, let $\bfA(z)$ be any $n\times n$ self-adjoint polyphase matrix whose diagonal entries are zero and whose off-diagonal entries are monomials over some abelian group $\calG$ of order $f\geq 2$.
If $\bfA(z)$ is an abelian $(n,f,c)$-DRACKN for some positive integer $c$, then evaluating \eqref{equation.DRACKN property} at any nontrivial character $\gamma$ of $\calG$ gives $[\bfA(\gamma)]^2=(n-fc-2)\bfA(\gamma)+(n-1)\bfI$,
implying that $\bfA(\gamma)$ is the signature matrix of an ETF for a space whose dimension $d$ is given by~\eqref{equation.dimension from signature} where $\delta=n-fc-2$.
Conversely, if $\bfA(\gamma)$ is such a signature matrix for every nontrivial character $\gamma$ of $\calG$, then since the only members of $\bbC\calG$ that vanish at every nontrivial character are scalar multiples of $\bfone(z)$,
there necessarily exists some $n\times n$ scalar-valued matrix $\bfB$ such that
$[\bfA(z)]^2=(n-fc-2)\bfA(z)+(n-1)\bfI+\bfone(z)\bfB$.
In this case, evaluating this equation at the trivial character gives
\begin{equation*}
(n-2)\bfJ+\bfI
=(\bfJ-\bfI)^2
=(n-fc-2)(\bfJ-\bfI)+(n-1)\bfI+f\bfB,
\end{equation*}
namely that $\bfB=c(\bfJ-\bfI)$, meaning $\bfA(z)$ satisfies~\eqref{equation.DRACKN property} and so is an abelian $(n,f,c)$-DRACKN.
We summarize these facts as follows:

\begin{lemma}
\label{lemma.DRACKN implies ETF}
For positive integers $n$, $f$ and $c$ with $f\geq2$ and an abelian group $\calG$ of order $f$,
let $\bfA(z)$ be an $n\times n$ self-adjoint polyphase matrix whose diagonal entries are zero and whose off-diagonal entries are monomials.  Then $\bfA(z)$ is an abelian $(n,f,c)$-DRACKN if and only if $\bfA(\gamma)$ is the signature matrix of an $n$-vector ETF for a space of dimension
\begin{equation}
\label{equation.dimension of abelian DRACKN ETF}
d=\frac n2\biggl[1-\frac{\delta}{\sqrt{\delta^2+4(n-1)}}\biggr],
\quad
\delta=n-fc-2,
\end{equation}
for all nontrivial characters $\gamma$ of $\calG$.
\end{lemma}

In~\cite{CoutinkhoGSZ16}, this idea is used to produce ETFs from known algebro-combinatorial constructions of abelian DRACKNs, as well as to derive necessary conditions on the existence of abelian DRACKNs from known necessary conditions on the existence of ETFs.
Nearly all of the examples of abelian $(n,f,c)$-DRACKNs given in~\cite{CoutinkhoGSZ16} have $n=fc$.
In this case, \eqref{equation.dimension of abelian DRACKN ETF} simplifies to $d=\frac12(n+\sqrt{n}\,)$.
This might be disappointing to researchers focused on finding new ETFs,
since other ETFs with these same $d$ and $n$ parameters are already well known.
Indeed, for any integer $v>3$, the Steiner ETF~\cite{FickusMT12} arising from all $2$-element subsets of $\set{1,\dotsc,v}$ consists of $v^2$ vectors which span a space of dimension $\frac12v(v-1)$;
a special case of these are Steiner ETFs arising from an affine geometry over the field $\bbZ_2$,
which themselves contain an infinite family of harmonic ETFs arising from a certain type of McFarland difference set~\cite{DingF07,JasperMF14}.
In particular, whenever $n$ is a perfect square, there exists a Steiner ETF whose Naimark complements have parameters $(d,n)$ where $d=\frac12(n+\sqrt{n}\,)$.
The only abelian $(n,f,c)$-DRACKN mentioned in~\cite{CoutinkhoGSZ16} that does not have $n=fc$ is a $(45,3,12)$-DRACKN from~\cite{KlinP11}.
By~\eqref{equation.dimension of abelian DRACKN ETF}, this DRACKN yields an ETF of $n=45$ vectors in a space of dimension $d=12$.
A harmonic ETF with these same parameters was already known to exist, arising from a McFarland difference set~\cite{DingF07}.

To be clear, it is difficult to determine whether the ETFs given in~\cite{CoutinkhoGSZ16} are new.
The signature matrices of two ETFs are usually considered \textit{equivalent} when one can be obtained from the other by conjugation by a phased permutation matrix.
As such, determining whether two ETFs are equivalent can be harder than determining whether two graphs are isomorphic.

In the coming sections, we use Lemma~\ref{lemma.DRACKN implies ETF} to construct an infinite family of complex ETFs that is demonstrably new, having new $d$ and $n$ parameters.
The main idea is to show that some members of a previously known class of DRACKNs happen to be abelian,
namely those that arise from abelian covers of BIBDs.

\section{Phased BIBD ETFs}

A BIBD is a type of finite geometry.
There is a well-known way to construct ETFs from BIBDs, namely the Steiner ETFs of~\cite{FickusMT12}.
In this section, we present a new way to construct ETFs from BIBDs.
As we shall see in the coming sections, this new construction is related to polyphase matrices and DRACKNs.

To be precise, for positive integers $v>k\geq2$ and $\lambda$,
a $\BIBD(v,k,\lambda)$ is a $v$-element vertex set $\calV$ along with a collection $\calB$ of subsets of $\calV$ called \textit{blocks} such that every block contains exactly $k$ vertices, and every pair of distinct vertices is contained in exactly $\lambda$ blocks.
Letting $b$ denote the number of blocks,
this means its $b\times v$ incidence matrix $\bfX$ satisfies $\bfX\bfone=k\bfone$ and that the off-diagonal entries of $\bfX^\rmT\bfX$ are all $\lambda$.
As such, for any $j=1,\dotsc,v$, the number $r_j$ of blocks that contains $j$th vertex satisfies
\begin{equation*}
(v-1)\lambda
=\sum_{\substack{j'=1\\j'\neq j}}^v(\bfX^\rmT\bfX)(j,j')
=\sum_{i=1}^{b}\bfX(i,j)\sum_{\substack{j'=1\\j'\neq j}}^v\bfX(i,j')
=\sum_{i=1}^{b}\left\{\begin{array}{cl}k-1,&\bfX(i,j)=1\\0,&\bfX(i,j)=0\end{array}\right\}
=r_j(k-1).
\end{equation*}
That is, this number $r_j=r:=\lambda\frac{v-1}{k-1}$ is independent of $j$.
At this point, summing all entries of $\bfX$ gives $bk=vr$ and so \smash{$b=\lambda\frac{v(v-1)}{k(k-1)}$} is also uniquely determined by $v$, $k$ and $\lambda$.
Moreover, since $v>k$ we have $r>\lambda$ and so $\bfX^\rmT\bfX=(r-\lambda)\bfI+\lambda \bfJ$ has full rank, implying $\bfX$ has rank $v$, and so $b\geq v$; this fact is known as \textit{Fisher's inequality}.

We focus on BIBDs where $\lambda=1$, which are also known as $(2,k,v)$-\textit{Steiner systems}.
In this case, the above facts about BIBDs can be summarized as
\begin{equation}
\label{equation.BIBD properties}
bk=vr,
\quad
v-1=r(k-1),
\quad
b\geq v,
\quad
\bfX\bfone=k\bfone,
\quad
\bfone^\rmT\bfX=r\bfone^\rmT,
\quad
\bfX^\rmT\bfX=(r-1)\bfI+\bfJ.
\end{equation}
In a $\BIBD(v,k,1)$, any two distinct vertices determine a unique block.
This in turn implies that any two distinct blocks have at most one vertex in common.
One classical example of such a design is an affine plane of order $q$, namely a $\BIBD(q^2,q,1)$ where $q$ is a prime power.

Now let $\bfPhi$ be any matrix whose entrywise squared-modulus $\abs{\bfPhi}^2$ equals the incidence matrix $\bfX$ of a $\BIBD(v,k,1)$.
(Here, squaring the modulus is unnecessary; we simply do it to be more consistent with theory introduced in the next section.)
That is, let $\bfPhi$ be a \textit{phased} $\BIBD(v,k,1)$.
Since any two columns of $\bfX$ have only one row index of common support, the columns $\set{\bfphi_i}_{i=1}^{n}$ of $\bfPhi$ are equiangular with $\norm{\bfphi_i}^2=r$ and $\abs{\ip{\bfphi_i}{\bfphi_j}}=1$ for all $i\neq j$.
This naturally leads one to ask when these columns also form a tight frame for their span:

\begin{definition}
If $\abs{\bfPhi}^2$ is the incidence matrix of a $\BIBD(v,k,1)$ and the columns of $\bfPhi$ form a tight frame for their span, we say they form a \textit{phased BIBD ETF}.
\end{definition}

Much of our work here is geared towards necessary conditions and explicit constructions of phased BIBD ETFs.
Since $b\geq v$, any phased BIBD $\bfPhi$ is a ``tall, skinny" matrix.
In order for the columns of $\bfPhi$ to form an ETF for their span, this means the rows of $\bfPhi$ should not be designed to be orthogonal (as is the case with harmonic ETFs and Steiner ETFs), but rather, so that $\bfPhi$ satisfies the criteria of Lemma~\ref{lemma.tight for subspace}.
To see that this is even possible, consider the following examples.

\begin{example}
\label{example.regular simplices as phased BIBD ETFs}
For any $v\geq3$, the collection of all $2$-element subsets of the vertex set $\set{1,\dotsc,v}$ forms a $\BIBD(v,2,1)$.
Here, $r=v-1$ and $b=\binom{v}2$.
Since $k=2$, each row of a corresponding incidence matrix $\bfX$ contains exactly two $1$'s.
Let $\bfPhi$ be obtained from $\bfX$ by negating one $1$ in each row of $\bfX$.
For example, when $v=3$, we can take
\begin{equation}
\label{equation.3x3 phased BIBD ETF}
\bfX
=\left[\begin{array}{rrr}1&1&0\\1&0&1\\0&1&1\end{array}\right],
\qquad
\bfPhi
=\left[\begin{array}{rrr}1&-1&0\\1&0&-1\\0&1&-1\end{array}\right].
\end{equation}
Then $\bfPhi^*\bfPhi=v\bfI-\bfJ$ and so satisfies $(\bfPhi^*\bfPhi)^2=(v\bfI-\bfJ)^2=v^2\bfI-v\bfJ=v\bfPhi^*\bfPhi$,
implying by Lemma~\ref{lemma.tight for subspace} that the columns of $\bfPhi$ form an ETF for their span of dimension $d=\frac{rv}{a}=r=v-1$.
Thus, the columns of $\bfPhi$ are a $v$-vector regular simplex.
In particular, the columns of the matrix $\bfPhi$ given in \eqref{equation.3x3 phased BIBD ETF} are a ``high"-dimensional representation of the famous \textit{Mercedes-Benz} ETF.
\end{example}

\begin{example}
\label{example.12x9 phased BIBD ETF}
Consider the following matrix obtained by phasing the incidence matrix of a $\BIBD(9,3,1)$ with a cube root of unity $z$,
as well as its corresponding Gram matrix:
\newlength{\NewLengthOne}
\settowidth{\NewLengthOne}{$z^2$}
\begin{equation}
\label{equation.12x9 phased BIBD ETF}
\setlength{\arraycolsep}{2pt}
\bfPhi(z)=\left[\begin{array}{lllllllll}
\makebox[\NewLengthOne][l]{1}&\makebox[\NewLengthOne][l]{1}&\makebox[\NewLengthOne][l]{1}&  0&  0&  0&  0&  0&  0\\
  0&  0&  0&  1&  1&  1&  0&  0&  0\\
  0&  0&  0&  0&  0&  0&  1&  1&  1\\
  1&  0&  0&  1&  0&  0&  1&  0&  0\\
  0&  1&  0&  0&z^2&  0&  0&  z&  0\\
  0&  0&  1&  0&  0&  z&  0&  0&z^2\\
  1&  0&  0&  0&  0&z^2&  0&z^2&  0\\
  0&  1&  0&  z&  0&  0&  0&  0&  1\\
  0&  0&  1&  0&  1&  0&  z&  0&  0\\
  1&  0&  0&  0&  z&  0&  0&  0&  z\\
  0&  1&  0&  0&  0&  1&z^2&  0&  0\\
  0&  0&  1&z^2&  0&  0&  0&  1&  0\\
\end{array}\right],
\quad
\bfPhi^*(z)\bfPhi(z)
=\left[\begin{array}{lllllllll}
  4&  1&  1&  1&  z&z^2&  1&z^2&  z\\
  1&  4&  1&  z&z^2&  1&z^2&  z&  1\\
  1&  1&  4&z^2&  1&  z&  z&  1&z^2\\
  1&z^2&  z&  4&  1&  1&  1&  z&z^2\\
z^2&  z&  1&  1&  4&  1&  z&z^2&  1\\
  z&  1&z^2&  1&  1&  4&z^2&  1&  z\\
  1&  z&z^2&  1&z^2&  z&  4&  1&  1\\
  z&z^2&  1&z^2&  z&  1&  1&  4&  1\\
z^2&  1&  z&  z&  1&z^2&  1&  1&  4
\end{array}\right].
\end{equation}
Since the above expression for the Gram matrix is valid for any complex $z$ such that $z^3=1$,
we can regard the entries of $\bfPhi(z)$ and $\bfPhi^*(z)\bfPhi(z)$ as members of the ring of polynomials $\bbC[z]/\gen{z^3-1}$, namely as the $z$-transforms of members of the group ring $\bbC\bbZ_3$.
A direct calculation reveals the product of the polyphase matrices in~\eqref{equation.12x9 phased BIBD ETF} is
\begin{equation}
\label{equation.12x19 polyphase equation}
\bfPhi(z)\bfPhi^*(z)\bfPhi(z)=6\bfPhi(z)+(1+z+z^2)(\bfJ-\bfX),
\end{equation}
where $\bfX$ is the incidence matrix of the underlying $\BIBD(9,3,1)$.
Applying the three characters to this expression equates to evaluating these polynomials at cube roots of unity.
In particular, evaluating~\eqref{equation.12x19 polyphase equation} at $\gamma=1$ gives
$\bfX\bfX^\rmT\bfX=6\bfX+3(\bfJ-\bfX)=3\bfX+3\bfJ$, a fact that also immediately follows from the properties of a BIBD given in~\eqref{equation.BIBD properties}.
Meanwhile, since $\gamma=\exp(\frac{2\pi\rmi}3)$ and its conjugate are roots of the geometric sum $1+z+z^2$,
evaluating~\eqref{equation.12x19 polyphase equation} at either of these points gives $\bfPhi(\gamma)\bfPhi^*(\gamma)\bfPhi(\gamma)=6\bfPhi(\gamma)$, implying by Lemma~\ref{lemma.tight for subspace} that the columns of $\bfPhi(\gamma)$ form an ETF for a subspace of $\bbC^{12}$ of dimension $6$.

We note that our $(d,n)=(6,9)$ ETF Gram matrix $\bfPhi^*(\gamma)\bfPhi(\gamma)$ here seems essentially the same as one constructed in~\cite{BodmannPT09}.
In~\cite{BodmannE10}, that construction was generalized to yield Gram matrices of ETFs with parameters $(d,n)=(\frac12q(q+1),q^2)$ whose off-diagonal elements are $q$th roots of unity, where $q\geq2$ is an arbitrary integer.
That level of generalization seems unlikely here, since the requisite $\BIBD(q^2,q,1)$ are affine planes of order $q$, which are famously conjectured to only exist when $q$ is a prime power.
\end{example}

Later on, we will revisit this example to gain greater insight into the relationship between certain ETFs, abelian DRACKNs and abelian GQs.
For now, having seen that phased BIBD ETFs indeed exist, we now state and prove what we were able to discover about them in general:

\begin{theorem}
\label{theorem.phased BIBD ETFs}
If the columns of $\bfPhi$ form an ETF for their span and $\abs{\bfPhi}^2$ is the $b\times v$ incidence matrix of a $\BIBD(v,k,1)$ then the rank of $\bfPhi$ is necessarily
\begin{equation}
\label{equation.phased BIBD ETF dimension}
d
=\frac{vr}{r+k-1},
\end{equation}
where $r=\frac{v-1}{k-1}$.
The redundancy $\frac{v}d$ of such an ETF is less than $2$.
Moreover, if $\bfPhi$ is a matrix such that $\abs{\bfPhi}^2$ is the $b\times v$ incidence matrix of a $\BIBD(v,k,1)$,
then the columns of $\bfPhi$ form an ETF for their span if and only if
\begin{equation}
\label{equation.phased BIBD ETF condition}
\sum_{i'=1}^{b}\sum_{j'=1}^{v}\bfPhi(i,j')\overline{\bfPhi(i',j')}\bfPhi(i',j)=0
\ \text{ for all $i$ and $j$ such that }\
\bfPhi(i,j)=0.
\end{equation}
\end{theorem}

\begin{proof}
Since $n=v$, $\norm{\bfphi_i}^2=r$ for all $i$, and $\abs{\ip{\bfphi_i}{\bfphi_j}}=1=w$ for all $i\neq j$, \eqref{equation.ETF parameter relations} gives
\begin{equation*}
\frac vd
=\frac{n-d}d+1
=\frac{(n-1)w^2}{r^2}+1
=\frac{v-1}{r^2}+1
=\frac{r+k-1}{r},
\end{equation*}
and so~\eqref{equation.phased BIBD ETF dimension}.
Moreover, $v\leq b$ and so $k\leq r$, implying $\frac vd=\frac{k-1}{r}+1<1+1=2$, as claimed.
For the final conclusion, note that by Lemma~\ref{lemma.tight for subspace},
the columns of $\bfPhi$ form an ETF for their span if and only if $\bfPhi\bfPhi^*\bfPhi=a\bfPhi$ where $a=\frac{rn}{d}=r\frac{v}d=r+k-1$.
That is, the columns of $\bfPhi$ form an ETF for their span if and only if
\begin{equation}
\label{equation.proof of phased BIBD ETFs 1}
(r+k-1)\bfPhi(i,j)
=(\bfPhi\bfPhi^*\bfPhi)(i,j)
=\sum_{i'=1}^{b}\sum_{j'=1}^{v}\bfPhi(i,j')\overline{\bfPhi(i',j')}\bfPhi(i',j),
\end{equation}
for all $i$ and $j$.
We claim this equation is automatically satisfied for those $i$ and $j$ for which $\abs{\bfPhi(i,j)}=1$.
Indeed, for any such $i$ and $j$, the only nonzero summands of \eqref{equation.proof of phased BIBD ETFs 1} occur when
$\abs{\bfPhi(i,j)}$, $\abs{\bfPhi(i,j')}$, $\abs{\bfPhi(i',j')}$ and $\abs{\bfPhi(i',j)}$ are all one,
namely when the $i$th and $i'$th vertices are contained in both the $j$th and $j'$th blocks.
Since two distinct vertices determine a unique block while any two distinct blocks have at most one vertex in common,
this happens precisely when either $i=i'$ or $j=j'$.
That is, when $\abs{\bfPhi(i,j)}=1$, the sum in~\eqref{equation.proof of phased BIBD ETFs 1} simplifies to
\begin{align*}
\bfPhi(i,j)\overline{\bfPhi(i,j)}\bfPhi(i,j)
+\sum_{\substack{j'=1\\j'\neq j}}^{v}\bfPhi(i,j')\overline{\bfPhi(i,j')}\bfPhi(i,j)
+\sum_{\substack{i'=1\\i'\neq i}}^{b}\bfPhi(i,j)\overline{\bfPhi(i',j)}\bfPhi(i',j)
=(r+k-1)\bfPhi(i,j),
\end{align*}
as claimed.
As such, the columns of $\bfPhi$ form a tight frame for their span if and only if \eqref{equation.proof of phased BIBD ETFs 1} holds for those $i$ and $j$ for which $\bfPhi(i,j)=0$.
\end{proof}

Note that only $k$ of the summands in~\eqref{equation.phased BIBD ETF condition} are nonzero:
since $\abs{\bfPhi}^2$ is the incidence matrix of a $\BIBD(v,k,1)$ and $\bfPhi(i,j)=0$, there are exactly $k$ choices of $j'\neq j$ such that $\bfPhi(i,j')\neq0$, and for each such $j'$, there is a unique choice of $i'$ such that $\bfPhi(i',j)\neq0\neq\bfPhi(i',j')$.
As such,~\eqref{equation.phased BIBD ETF condition} means that whenever $\bfPhi(i,j)=0$,
the $k$ nonzero values of $\bfPhi(i,j')\overline{\bfPhi(i',j')}\bfPhi(i',j)$ are equally distributed about the origin.

Further note that in order for a phased BIBD ETF to exist, the dimension~\eqref{equation.phased BIBD ETF dimension} of their span is necessarily an integer.
In the next result, we relate this fact to other known necessary conditions on the existence of ETFs.
Some of these conditions only apply when the ETF in question is real, such as the regular simplices discussed in Example~\ref{example.regular simplices as phased BIBD ETFs}.
In particular, Theorem~A of~\cite{SustikTDH07} states that if there exists a real $n$-vector ETF for a space of dimension $d$ and $1<d<n-1$ with $n\neq 2d$ then both
\begin{equation*}
\biggl[\frac{d(n-1)}{n-d}\biggr]^{\frac12},
\qquad
\biggl[\frac{(n-d)(n-1)}{d}\biggr]^{\frac12}
\end{equation*}
are necessarily odd integers.
This follows from a delicate analysis of the spectrum of its signature matrix.
Besides this, the only known necessary conditions on the existence of ETFs are the \textit{Gerzon bounds},
which state that if an $n$-vector ETF for a $d$-dimensional Hilbert space $\bbH_d$ exists where $d>1$, then $n\leq\binom{d+1}2$ if the ETF is real and $n\leq d^2$ if the ETF is complex.
Indeed, for any noncollinear equiangular vectors $\set{\bfphi_i}_{i=1}^{n}$,
the corresponding operators $\set{\bfphi_i^{}\bfphi_i^*}_{i=1}^{n}$ lie in the real space of all self-adjoint operators from $\bbH_d$ to itself, which is a Hilbert space under the Frobenius-Hilbert-Schmidt inner product $\ip{\bfA}{\bfB}:=\Tr(\bfA^*\bfB)$;
this space has dimension $\binom{d+1}2$ or $d^2$ depending on whether these vectors are real or complex,
and these operators are linearly independent since their Gram matrix $\abs{\bfPhi^*\bfPhi}^2=(r^2-w^2)\bfI+w^2\bfJ$ is invertible.
Since any phased BIBD ETF has $v\leq 2d$, the Gerzon bounds are best applied to its Naimark complements:
if $v>d+1$, a phased BIBD ETF has $v\leq\binom{v-d+1}2$ if it is real and $v\leq(v-d)^2$ if it is complex.

\begin{theorem}
\label{theorem.phased BIBD necessary conditions}
If the columns of $\bfPhi$ form an ETF for their span and $\abs{\bfPhi}^2$ is the $b\times v$ incidence matrix of a $\BIBD(v,k,1)$ then
\begin{equation}
\label{equation.phased BIBD ETF parameter}
u:=\frac{k(k-1)^2(k-2)}{v+k(k-2)}
\end{equation}
is a nonnegative integer that divides both $k(k-1)(k-2)$ and $r(k-1)(k-2)$ where \smash{$r=\frac{v-1}{k-1}$}.
Also, $u\leq\frac12(k-1)(k-2)$.
Moreover, under these assumptions,
\begin{equation}
\label{equation.phased BIBD ETF reciprocal Welch bound}
\biggl[\frac{d(v-1)}{v-d}\biggr]^{\frac12}=r,
\qquad
\biggl[\frac{(v-d)(v-1)}{d}\biggr]^{\frac12}=k-1,
\end{equation}
where $d$ is given by~\eqref{equation.phased BIBD ETF dimension}.
In particular, if $k>2$ and $\bfPhi$ is real then $v$ and $k$ are even, and $r$ is odd.

Finally, various types of phased BIBD ETFs are characterized by the value of their $u$ parameter:
\begin{enumerate}
\alphi
\item $u=0\ \Leftrightarrow\ k=2\ \Leftrightarrow\ v=d+1$,
\item $u=1\ \Leftrightarrow\ v=k^2(k-2)^2\ \Leftrightarrow\ v=(v-d)^2$,
\item $u=2\ \Leftrightarrow\ v=\binom{k(k-2)}{2}\ \Leftrightarrow\ v=\binom{v-d+1}2$,
\item $u=\frac12(k-1)(k-2),\, k>2\ \Leftrightarrow\ v=k^2,\, k>2\ \Leftrightarrow\ v-d=\frac12(v-v^{\frac12})$, $v>4$.
\end{enumerate}
\end{theorem}

The proof of this result is given in Appendix~A.
One consequence of it is that for any fixed $k>2$,
there are only a finite number of choices of $v$ for which a phased $\BIBD(v,k,1)$ ETF exists.
In Table~\ref{table.phased BIBD ETFs}, we list the BIBD ETF parameters that meet the necessary conditions of Theorem~\ref{theorem.phased BIBD necessary conditions} for $3\leq k\leq 9$.
We also list the dimension $v-d$ of the span of a Naimark complement of a corresponding phased BIBD ETF, should it exist.
We list this dimension instead of $d$ since phased BIBD ETFs have $v<2d$, cf.\ Theorem~\ref{theorem.phased BIBD ETFs},
whereas other tables of ETFs such as~\cite{FickusM15} assume $v>2d$.
As seen in the proof above, these necessary conditions ensure that $r=\frac{v-1}{k-1}$ and $b=\frac vk r$ are integers and that Fisher's inequality ($v\leq b$) is satisfied.
These necessary conditions do not take into account other known necessary conditions on the existence of $\BIBD(v,k,1)$, such as the Bruck-Ryser-Chowla Theorem.
This is why this table lists the parameters $(v,k)=(36,6)$ despite the fact that no affine plane of order $6$ exists; this means that known ETFs with $(d,n)=(21,36)$ cannot be realized as a phased $\BIBD(36,6,1)$.

\newlength{\NewLengthTwo}
\settowidth{\NewLengthTwo}{218736}
\begin{table}
\begin{center}
\begin{tabular}{ccccccccc}
\makebox[\NewLengthTwo]{$v-d$}&\makebox[\NewLengthTwo]{$v$}&\makebox[\NewLengthTwo]{$k$}&\makebox[\NewLengthTwo]{$r$}&\makebox[\NewLengthTwo]{$b$}&\makebox[\NewLengthTwo]{$u$}%
&\makebox[\NewLengthTwo]{ETF?}&\makebox[\NewLengthTwo]{BIBD?}&\makebox[\NewLengthTwo]{PBETF?}\\
\hline
    3&     9&     3&     4&    12&     1&$\bbC$&     Y&$\bbC$\\
\hline
    6&    16&     4&     5&    20&     3&$\bbR$&     Y&$\bbR$\\
    7&    28&     4&     9&    63&     2&$\bbR$&     Y&$\bbR$\\
    8&    64&     4&    21&   336&     1&$\bbC$&     Y&      \\
\hline
   10&    25&     5&     6&    30&     6&$\bbC$&     Y&$\bbC$\\
   12&    45&     5&    11&    99&     4&$\bbC$&     Y&      \\
   13&    65&     5&    16&   208&     3&$\bbC$&     Y&$\bbC$\\
   14&   105&     5&    26&   546&     2&      &     Y&      \\
   15&   225&     5&    56&  2520&     1&$\bbC$&      &      \\
\hline
   15&    36&     6&     7&    42&    10&$\bbR$&     N&     N\\
   17&    51&     6&    10&    85&     8&      &      &      \\
   19&    76&     6&    15&   190&     6&$\bbC$&     Y&      \\
   20&    96&     6&    19&   304&     5&$\bbC$&     Y&      \\
   21&   126&     6&    25&   525&     4&$\bbR$&     Y&$\bbR$\\
   23&   276&     6&    55&  2530&     2&$\bbR$&      &      \\
   24&   576&     6&   115& 11040&     1&$\bbC$&      &      \\
\hline
   21&    49&     7&     8&    56&    15&$\bbC$&     Y&$\bbC$\\
   26&    91&     7&    15&   195&    10&$\bbC$&     Y&      \\
   30&   175&     7&    29&   725&     6&$\bbC$&     Y&      \\
   31&   217&     7&    36&  1116&     5&$\bbC$&     Y&      \\
   33&   385&     7&    64&  3520&     3&      &      &      \\
   34&   595&     7&    99&  8415&     2&      &      &      \\
   35&  1225&     7&   204& 35700&     1&$\bbC$&      &      \\
\hline
   28&    64&     8&     9&    72&    21&$\bbR$&     Y&$\bbR$\\
   35&   120&     8&    17&   255&    14&$\bbR$&     Y&      \\
   42&   288&     8&    41&  1476&     7&      &     Y&      \\
   43&   344&     8&    49&  2107&     6&$\bbR$&     Y&$\bbR$\\
   46&   736&     8&   105&  9660&     3&      &      &      \\
   47&  1128&     8&   161& 22701&     2&      &      &      \\
   48&  2304&     8&   329& 94752&     1&$\bbC$&      &      \\
\hline
   36&    81&     9&    10&    90&    28&$\bbC$&     Y&$\bbC$\\
   50&   225&     9&    28&   700&    14&$\bbC$&      &      \\
   56&   441&     9&    55&  2695&     8&$\bbC$&      &      \\
   57&   513&     9&    64&  3648&     7&$\bbC$&     Y&$\bbC$\\
   60&   945&     9&   118& 12390&     4&      &      &      \\
   62&  1953&     9&   244& 52948&     2&      &      &      \\
   63&  3969&     9&   496&218736&     1&      &      &      \\
\end{tabular}
\caption{%
\label{table.phased BIBD ETFs}
The parameters of all $\BIBD(v,k,1)$ with $3\leq k\leq 9$ that meet the necessary conditions on phased BIBD ETFs given in Theorem~\ref{theorem.phased BIBD necessary conditions}.
BIBDs with $k=2$ are not listed since they correspond to regular simplices, cf.\ Example~\ref{example.regular simplices as phased BIBD ETFs} and Theorem~\ref{theorem.phased BIBD necessary conditions}(a).
For each $v$ and $k$, we list the $r$ and $b$ BIBD parameters as well as the $u$ parameter~\eqref{equation.phased BIBD ETF parameter}.
We also list the dimension $v-d$ of the span of the Naimark complements of a corresponding phased BIBD ETF, should it exist.
The ``ETF?" column lists whether any $v$-vector real or complex ETF for a space of dimension $v-d$ is known to exist; see~\cite{FickusM15} and also see~\cite{FickusMJ16} for the $(v-d,v)=(19,76)$ case.
The ``BIBD?" column indicates whether such a BIBD is known to exist~\cite{ColbournM07,MathonR07}.
The ``PBETF?" column indicates whether a real or complex phased BIBD ETF with these parameters is known;
all known examples are constructed in this paper.
Blank entries are unknown to us.
}
\end{center}
\end{table}

Note that $u=1$ meets all the necessary conditions of Theorem~\ref{theorem.phased BIBD necessary conditions} for any $k\geq 3$.
As seen in (b) of Theorem~\ref{theorem.phased BIBD necessary conditions},
this choice of $u$ corresponds to ETFs whose Naimark complements achieve the complex Gerzon bound.
As such, this approach might lead to new constructions of ETFs with $v=d^2$, called
\textit{symmetric, informationally complete, positive operator--valued measures} (SIC-POVMs) in the quantum information theory literature~\cite{RenesBSC04,Zauner99}.
Little seems to be known about BIBDs with $v=k^2(k-2)^2$ in general, and so the only Naimark complement of a SIC-POVM that we were able to construct as a phased BIBD ETF is given in Example~\ref{example.12x9 phased BIBD ETF}.
Similarly, $u=2$ meets all necessary conditions for any $k\geq 4$,
and by (c) any real phased BIBD with $v=\binom{k(k-2)}{2}$ would achieve the real Gerzon bound;
such ETFs are equivalent to tight spherical $5$-designs~\cite{FickusM15}.

In the coming sections, we show how to construct a phased $\BIBD(k^2,k,1)$ ETF whenever $k$ is the power of a prime.
Unfortunately, by Theorem~\ref{theorem.phased BIBD necessary conditions}(d), these ETFs are not necessarily new, since they are of the same size of Naimark complements of Steiner ETFs arising from $\BIBD(k,2,1)$~\cite{FickusMT12}.
We also show how to construct a phased $\BIBD(q^3+1,q+1,1)$ ETF for any prime power $q$.
Here,~\eqref{equation.phased BIBD ETF parameter} becomes $u=q-1$,
which clearly divides $k(k-1)(k-2)=q(q^2-1)$ and $r(k-1)(k-2)=q^3(q-1)$ and is no more than $\frac12(k-1)(k-2)=\frac12q(q-1)$.
The Naimark complement of the resulting ETF consists of $n=v=q^3+1$ vectors that span a space of dimension \smash{$v-d=(k-1)^2-u=q^2-q+1=\frac{q^3+1}{q+1}$}.
Whenever $q$ is an odd prime power, real ETFs of this size are already known to exist,
since they correspond to known SRGs constructed in~\cite{Godsil92}.
We show how the synthesis operator of such an ETF can be represented as a phased $\BIBD(q^3+1,q+1,1)$.
Moreover, we show that whenever $q$ is an even prime power, this same construction produces a complex phased BIBD ETF with $n=q^3+1$, \smash{$n-d=\frac{q^3+1}{q+1}$}.
As we shall explain, these complex ETFs are new whenever $q$ is an even prime power with the property that $q-1$ is not an odd prime power, which happens infinitely often.
For both of these constructions, the key idea is to not regard the entries of a phased BIBD as complex numbers, but rather as polynomials over some finite abelian group, as done in Example~\ref{example.12x9 phased BIBD ETF}.

\section{Polyphase BIBD ETFs}

In this section, we lay the foundation for two new constructions of phased BIBD ETFs.
Like Example~\ref{example.12x9 phased BIBD ETF}, we actually construct a polyphase matrix $\bfPhi(z)$ that gives an ETF when evaluated at any nontrivial character $\gamma$:
\begin{definition}
\label{definition.polyphase BIBD ETF}
Let $\calG$ be an abelian group of order $f$.
We say a polyphase matrix $\bfPhi(z)$ whose entries are either monomials over $\calG$ or zero is a \textit{(v,k,f)-polyphase BIBD ETF} if both:
\begin{enumerate}
\romani
\item
$\abs{\bfPhi(z)}^2$ is the incidence matrix of a $\BIBD(v,k,1)$,
\item
for all nontrivial characters $\gamma$ of $\calG$, the columns of $\bfPhi(\gamma)$ form an ETF for their span.
\end{enumerate}
\end{definition}

Here, the entrywise modulus squared $\abs{\bfPhi(z)}^2$ of any polyphase matrix $\bfPhi(z)$ is defined as the matrix whose $(i,j)$th entry is $\abs{\bfx(z)}^2=\overline{\bfx(z)}\bfx(z)=\tilde{\bfx}(z)\bfx(z)$ where $\bfx(z)$ is the $(i,j)$th entry of $\bfPhi(z)$.
We now characterize polyphase BIBD ETFs and relate them to abelian DRACKNs.

\begin{theorem}
\label{theorem.polyphase BIBD ETF}
Let $\calG$ be an abelian group of order $f$,
and let $\bfPhi(z)$ be a polyphase matrix whose entries are either monomials over $\calG$ or zero with the property that $\bfX=\abs{\bfPhi(z)}^2$ is the incidence matrix of a $\BIBD(v,k,1)$.
Letting $r:=\frac{v-1}{k-1}$, $\bfone(z):=\sum_{g\in\calG}z^g$ and letting $\bfPhi(z;i,j)$ denote the $(i,j)$th entry of $\bfPhi(z)$, the following are equivalent:\smallskip
\begin{enumerate}
\romani
\item
$\bfPhi(z)$ is a polyphase BIBD ETF, cf.\ Definition~\ref{definition.polyphase BIBD ETF}.
\medskip
\item
$\bfPhi(z)\bfPhi^*(z)\bfPhi(z)=(r+k-1)\bfPhi(z)+\tfrac kf\bfone(z)(\bfJ-\bfX)$.
\item
$\displaystyle\sum_{i'=1}^{b}\sum_{j'=1}^{v}\bfPhi(z;i,j')\overline{\bfPhi(z;i',j')}\bfPhi(z;i',j)=\tfrac kf\bfone(z)$ for all $i$ and $j$ such that $\bfPhi(z;i,j)=0$.
\item
For all $i$ and $j$ such that $\bfPhi(z;i,j)=0$, the cardinality of
\begin{equation*}
\set{(i',j'): z^g = \bfPhi(z;i,j')\overline{\bfPhi(z;i',j')}\bfPhi(z;i',j)}
\end{equation*}
is constant over all $g\in\calG$.
\end{enumerate}
In this case $f$ necessarily divides $k$ and $\bfA(z)=\bfPhi^*(z)\bfPhi(z)-r\bfI$ is an abelian $(v,f,c)$-DRACKN with $c=\tfrac{k(r-1)}f$, cf.\ Definition~\ref{definition.abelian DRACKN}.
Moreover, if $f$ is even then $\gamma$ can be chosen so that $\bfPhi(\gamma)$ is a real ETF.
\end{theorem}

\begin{proof}
(i $\Rightarrow$ ii)
Let $\gamma_0$ denote the trivial character of $\calG$.
For any $\gamma\neq\gamma_0$,
applying Lemma~\ref{lemma.tight for subspace} to $\bfPhi(\gamma)$ gives $\bfPhi(\gamma)\bfPhi^*(\gamma)\bfPhi(\gamma)=a_\gamma\bfPhi(\gamma)$ for some $a_\gamma=\frac{vr}{d_\gamma}$,
where $d_\gamma$ is the rank of $\bfPhi(\gamma)$.
From Theorem~\ref{theorem.phased BIBD ETFs}, we know $d_\gamma=d=\frac{vr}{r+k-1}$ and so $a_\gamma=r+k-1$ for all $\gamma\neq\gamma_0$.
As such, evaluating $\bfPhi(z)\bfPhi^*(z)\bfPhi(z)-(r+k-1)\bfPhi(z)$ at any $\gamma\neq\gamma_0$ gives $\bfzero$.
Now note that a $z$-transform $\bfx(z)$ satisfies $\bfx(\gamma)=0$ for all $\gamma\neq\gamma_0$ precisely when $\bfx(z)$ is a scalar multiple of the geometric sum $\bfone(z)$.
This is because the characters form an orthonormal basis for $\bbC^\calG$, and so $\ip{\gamma}{\bfx}=0$ for all $\gamma\neq\gamma_0$ if and only if $\bfx$ is a scalar multiple of $\gamma_0$.
As such, there exists a matrix $\bfX_0$ with scalar (constant polynomial) entries such that
$\bfPhi(z)\bfPhi^*(z)\bfPhi(z)=(r+k-1)\bfPhi(z)+\bfone(z)\bfX_0$.
Here, $\bfX_0$ is uniquely determined by evaluating this equation at $\gamma_0$:
since $\bfPhi(\gamma_0)=\bfX$ while $\bfone(\gamma_0)=f$, \eqref{equation.BIBD properties} gives
\begin{equation*}
(r+k-1)\bfX+f\bfX_0
=\bfX\bfX^\rmT\bfX
=\bfX[(r-1)\bfI+\bfJ]
=(r-1)\bfX+k\bfJ,
\end{equation*}
namely that $\bfX_0=\frac kf(\bfJ-\bfX)$.
Thus $\bfPhi(z)\bfPhi^*(z)\bfPhi(z)=(r+k-1)\bfPhi(z)+\tfrac kf\bfone(z)(\bfJ-\bfX)$, as claimed.

(ii $\Rightarrow$ iii)
For any $(i,j)$ such that $\bfPhi(z;i,j)=0$, simply compute the $(i,j)$th entry of (ii).

(iii $\Rightarrow$ i)
For any $\gamma\neq\gamma_0$, $\bfone(\gamma)=0$ and so $\bfPhi=\bfPhi(\gamma)$ is a phased BIBD which satisfies~\eqref{equation.phased BIBD ETF condition},
and so Theorem~\ref{theorem.phased BIBD ETFs} gives that its columns form an ETF for their span.

(iii $\Leftrightarrow$ iv)
Since $\abs{\bfPhi(z)}^2$ is a $\BIBD(v,k,1)$, for any $(i,j)$ such that $\bfPhi(i,j)=0$, there are exactly $k$ choices of $(i',j')$ such that $\bfPhi(z;i,j')\overline{\bfPhi(z;i',j')}\bfPhi(z;i',j)$ is nonzero:
there are $k$ choices of $j'$ such that $\bfPhi(z;i,j')\neq0$,
and for each there is a unique choice of $i'$ such that $\bfPhi(i',j)\neq0\neq\bfPhi(i',j')$.
Moreover, each $\bfPhi(z;i,j')\overline{\bfPhi(z;i',j')}\bfPhi(z;i',j)$ is a monomial $z^g$ over our group of order $f$, being the product of three monomials.
As such, (iii) holds precisely when each monomial $z^g$ appears as one of the $k$ nonzero values  $\bfPhi(z;i,j')\overline{\bfPhi(z;i',j')}\bfPhi(z;i',j)$ exactly $\frac kf$ times, namely (iv).

Now assume (i)--(iv) hold.
Note $f$ divides $k$ since $\frac kf$ is the cardinality of the sets in (iv).
To show $\bfA(z)=\bfPhi^*(z)\bfPhi(z)-r\bfI$ is an abelian DRACKN, note that $\bfA(z)$ is self-adjoint.
Moreover, the BIBD structure of $\bfX$ implies that the diagonal entries of $\bfA(z)$ are zero while its off-diagonal entries are monomials, being the product of two monomials.
What remains to be shown is that $\bfA(z)$ satisfies the quadratic-like equation given in Definition~\ref{definition.abelian DRACKN}.
To do so, note that multiplying (ii) by $\bfPhi^*(z)$ gives
\begin{equation*}
[\bfPhi^*(z)\bfPhi(z)]^2=(r+k-1)\bfPhi^*(z)\bfPhi(z)+\tfrac kf\bfone(z)\bfPhi^*(z)(\bfJ-\bfX).
\end{equation*}
To simplify, note that for any formal polynomial $\bfx(z)$ over $\calG$, $\bfx(z)\bfone(z)=[\sum_{g\in\calG}\bfx(g)]\bfone(z)$.
Thus, $\bfone(z)\bfPhi^*(z)=\bfone(z)\bfX^\rmT$.
Since~\eqref{equation.BIBD properties} gives $\bfX^\rmT(\bfJ-\bfX)=r\bfJ-[(r-1)\bfI+\bfJ]=(r-1)(\bfJ-\bfI)$, the above equation simplifies to
\begin{equation*}
[\bfPhi^*(z)\bfPhi(z)]^2=(r+k-1)\bfPhi^*(z)\bfPhi(z)+\tfrac{k(r-1)}f\bfone(z)(\bfJ-\bfI).
\end{equation*}
Replacing $\bfPhi^*(z)\bfPhi(z)$ with $\bfA(z)+r\bfI$ and simplifying then gives
\begin{equation*}
[\bfA(z)]^2
=-(r-k+1)\bfA(z)+(v-1)\bfI+\tfrac{k(r-1)}f\bfone(z)(\bfJ-\bfI).
\end{equation*}
Letting $c=\tfrac{k(r-1)}f$ and comparing the above expression to Definition~\ref{definition.abelian DRACKN},
we observe that
\begin{equation}
\label{equation.polphase BIBD ETF delta}
n-fc-2=(v-1)-k(r-1)-1=r(k-1)-k(r-1)-1=-(r-k+1),
\end{equation}
and so conclude that $\bfA(z)$ is a $(v,f,c)$-DRACKN.
For the final conclusion, note that if $f$ is even, then $\calG$ has a real-valued character $\gamma$,
meaning $\bfPhi(\gamma)$ is real.
\end{proof}

We remark that the DRACKN parameter $\delta=n-fc-2=-(r+k+1)$ computed in~\eqref{equation.polphase BIBD ETF delta} appears frequently in~\cite{GodsilH92,CoutinkhoGSZ16}.
In particular, as summarized in Lemma~\ref{lemma.DRACKN implies ETF}, $\bfA(\gamma)$ is the signature matrix of an ETF for a space whose dimension is determined by $n=v$ and $\delta$:
\begin{equation*}
d
=\frac v2\biggl[1-\frac{\delta}{\sqrt{\delta^2+4(v-1)}}\biggr]
=\frac v2\biggl[1+\frac{r-k+1}{r+k-1}\biggr]
=\frac{vr}{r+k-1}.
\end{equation*}
This expression for $d$ also follows from Theorem~\ref{theorem.phased BIBD ETFs} since $\bfA(\gamma)$ is the signature matrix of the phased BIBD ETF $\bfPhi(\gamma)$.

Many of the known constructions of DRACKNs have $\delta\in\set{-2,0,2}$~\cite{GodsilH92}.
For DRACKNs produced by Theorem~\ref{theorem.polyphase BIBD ETF},
recall Fisher's inequality gives $r\geq k$ and so $\delta=-(r-k+1)$ is negative.
In fact, as explained in the proof of Theorem~\ref{theorem.phased BIBD necessary conditions}, the only phased $\BIBD(v,k,1)$ ETFs with $r=k$ have $(v,k)=(3,2)$, such as~\eqref{equation.3x3 phased BIBD ETF}.
Thus, excluding this case, the DRACKNs produced by polyphase BIBD ETFs have $\delta\leq -2$.
Moreover, $\delta=-2$ precisely when $r=k+1$, namely when Theorem~\ref{theorem.phased BIBD necessary conditions}(d) applies to $\bfPhi(\gamma)$ for all nontrivial $\gamma$.

Some of the deepest results of~\cite{CoutinkhoGSZ16} provide necessary conditions on when an ETF arising from an abelian DRACKN can possibly achieve the real or complex Gerzon bounds.
In particular, since polyphase BIBD ETFs have $\delta<0$, Theorem~6.5 and Corollary~6.7 of~\cite{CoutinkhoGSZ16} imply that any polyphase BIBD ETF for which $v=(v-d)^2$ necessarily has $(v,f,c)=(9,3,3)$, namely $(v,k,f)=(9,3,3)$.
That is, Example~\ref{example.12x9 phased BIBD ETF} is essentially the only polyphase BIBD ETF whose Naimark complements achieve the complex Gerzon bound.
Meanwhile, Theorems~6.5 and~6.6 of~\cite{CoutinkhoGSZ16} imply that any $(v,k,f)$-polyphase BIBD ETF for which $v=\binom{v-d+1}2$ and $f>2$ necessarily has $(v,f,c)=(28,4,8)$, namely $(v,k,f)=(28,4,4)$;
later on, we show how to explicitly construct such a polyphase BIBD ETF over the group $\calG=\bbZ_4$.
We note that~\cite{CoutinkhoGSZ16} also gives necessary conditions on DRACKNs with $\delta>0$ that attain the Gerzon bounds, and these conditions are much less restrictive than those given in the $\delta<0$ case.

We also remark that from the perspective of~\cite{GodsilH92,CoutinkhoGSZ16},
a polyphase $\BIBD(v,k,1)$ ETF $\bfPhi(z)$ can be naturally interpreted as an ``abelian cover" of that BIBD.
To elaborate, under the isomorphism~\eqref{equation.definition of x(T)} between polynomials and circulant matrices over $\calG$, the filter bank $\bfY=\bfPhi(\bfT)$ is a $b\times v$ array of $f\times f$ blocks.
When $\bfPhi(z;i,j)=0$, the corresponding $(i,j)$th block of $\bfY$ is $\bfY(i,j)=\bfzero_f$.
When $\bfPhi(z;i,j)=z^g$, $\bfY(i,j)=\bfT^g$ is a circulant permutation matrix, namely a ``perfect matching" between the vertices in the $i$th ``vertex fiber" and the $j$th ``block fiber."
This implies $\bfY$ is a $bf\times vf$ incidence matrix.
Moreover, $\bfone(z)$ becomes $\bfone(\bfT)=\bfJ_f$ under~\eqref{equation.definition of x(T)},
and the characterization of Theorem~\ref{theorem.polyphase BIBD ETF}(ii) becomes
\begin{equation*}
\bfY\bfY^\rmT\bfY=(r+k-1)\bfY+\tfrac kf(\bfJ_{b\times v}-\bfX)\otimes\bfJ_f.
\end{equation*}
This implies that if a given vertex does not lie on a given block nor on any other block in its fiber then there are exactly $\frac kf$ projections from that vertex onto that block.
As we now explain, this means that polyphase BIBD ETFs with $f=k$ are closely related to combinatorial designs known as generalized quadrangles.

\subsection{Polyphase BIBD ETFs from abelian generalized quadrangles}
\label{subsection.GQ}

Given positive integers $s$ and $t$, a corresponding \textit{generalized quadrangle} $\GQ(s,t)$ is a set of vertices and a set of blocks (subsets of the vertex set) such that:
\begin{enumerate}
\romani
\item every block contains exactly $s+1$ vertices,
\item every vertex is contained in exactly $t+1$ blocks,
\item two distinct blocks intersect in at most one vertex,
\item two distinct vertices are contained in at most one block,
\item if vertex $i$ does not lie in block $j$, then there exists a unique $(i',j')$ such that vertex $i'$ is contained in block $j$ and $j'$ while block $j'$ contains both vertex $i$ and $i'$.
\end{enumerate}
In particular, the first two axioms mean the GQ's incidence matrix $\bfZ$ satisfies $\bfZ\bfone=(s+1)\bfone$ and $\bfZ^\rmT\bfone=(t+1)\bfone$.
The next two axioms state that all off-diagonal entries of $\bfZ\bfZ^\rmT$ and $\bfZ^\rmT\bfZ$ are $\set{0,1}$-valued.
The final axiom means every vertex not on a block has a unique ``projection" onto that block.
This implies a GQ contains no triangles.
To express (v) in terms of $\bfZ$, note
\begin{equation*}
(\bfZ\bfZ^\rmT\bfZ)(i,j)=\sum_{i'}\sum_{j'}\bfZ(i,j')\bfZ(i',j')\bfZ(i',j)
\end{equation*}
counts the number of vertex-block pairs $(i',j')$ such that $\bfZ(i,j')=\bfZ(i',j')=\bfZ(i',j)=1$.
Thus, (v) states that if $\bfZ(i,j)=0$ then $(\bfZ\bfZ^\rmT\bfZ)(i,j)=1$.
If we instead have $\bfZ(i,j)=1$, the first four axioms imply there are exactly $s+t+1$ choices of $(i',j')$ such that $\bfZ(i,j')=\bfZ(i',j')=\bfZ(i',j)=1$,
namely those $s+1$ choices of $(i,j')$ such that $\bfZ(i,j')=1$ and those $t+1$ choices of $(i',j)$ such that $\bfZ(i,j')$, both of which include $(i',j')=(i,j)$.
Overall, we see that a $\GQ(s,t)$ is equivalent to an incidence matrix $\bfZ$ that satisfies
\begin{align}
\label{equation.definition of GQ 1}
&\bfZ\bfone=(s+1)\bfone,\quad\bfZ^\rmT\bfone=(t+1)\bfone,\ \forall i,j,\\
\label{equation.definition of GQ 2}
&(\bfZ\bfZ^\rmT)(i,i'),\,(\bfZ^\rmT\bfZ)(j,j')\in\set{0,1},\ \forall i\neq i', j\neq j',\\
\label{equation.definition of GQ 3}
&\bfZ\bfZ^\rmT\bfZ=(s+t)\bfZ+\bfJ.
\end{align}

This formulation leads to several well-known results about GQs that we will need later on.
For example, we see that the dual of a $\GQ(s,t)$ (obtained by transposing $\bfZ$) is a $\GQ(t,s)$.
Also, letting $\bfZ$ be $b\times v$, multiplying $\bfZ\bfZ^\rmT\bfZ=(s+t)\bfZ+\bfJ$ by $\bfone$ gives $v=(s+1)(st+1)$,
at which point the dual result gives $b=(t+1)(st+1)$.
Next, \eqref{equation.definition of GQ 1} and~\eqref{equation.definition of GQ 3} immediately gives
\begin{equation}
\label{equation.square of GQ Gram}
(\bfZ^\rmT\bfZ)^2
=(s+t)\bfZ^\rmT\bfZ+(t+1)\bfJ,
\end{equation}
implying the adjacency matrix $\bfZ^\rmT\bfZ-(t+1)\bfI$ of its \textit{collinearity graph} is an SRG~\eqref{equation.definition of SRG} with parameters
\begin{equation}
\label{equation.DRACKN SRG parameters}
(v,k,\lambda,\mu)=\bigl((s+1)(st+1),s(t+1),s-1,t+1\bigr).
\end{equation}
One of the eigenvalues of this adjacency matrix thus has multiplicity \smash{$\frac1{s+t}st(s+1)(t+1)$}, and so $s+t$ necessarily divides $st(s+1)(t+1)$.
See~\cite{Payne07,PayneT09} for more necessary conditions on GQ parameters, such as the fact that $s\leq t^2$ when $t\neq1$ and $t\leq s^2$ when $s\neq1$.

In light of these restrictions, it is not surprising that relatively few examples of GQs are known:
they are only known to exist for $(s,t)$ or $(t,s)$ of the form
\begin{equation}
\label{equation.parameters of known GQ}
(1,r),
\quad
(q,q),
\quad
(q,q^2),
\quad
(q^2,q^3),
\quad
(q-1,q+1)
\end{equation}
for any $r\geq 1$ and any prime power $q$~\cite{Payne07}.
These known constructions all involve special algebro-combinatorial structures over finite fields.
The existence of $\GQ(s,t)$ remains unresolved for many values of $(s,t)$, such as $(s,t)=(4,11)$.

In certain GQs, there is a collection of blocks that partition the vertex set.
Such a collection is called a \textit{spread}~\cite{deBruyn02}.
Since any spread necessarily consists of $st+1$ disjoint blocks,
this happens precisely when the GQ has an incidence matrix (and resulting collinearity adjacency matrix) of the form
\begin{equation*}
\bfZ
=\left[\begin{array}{cc}\bfI_{st+1}^{}\otimes\bfone_{s+1}^\rmT\\\bfY\end{array}\right],
\qquad
\bfZ^\rmT\bfZ-(t+1)\bfI
=(\bfY^\rmT\bfY^{}-t\bfI)+\bfI_{st+1}\otimes(\bfJ_{s+1}-\bfI_{s+1}).
\end{equation*}
The corresponding SRG can thus be partitioned into cliques of size $s+1$.
As noted in~\cite{Brouwer84,GodsilH92},
deleting the edges in these cliques produces the adjacency matrix $\bfY^\rmT\bfY^{}-t\bfI$ of a DRACKN with parameters $(n,f,c)=(st+1,s+1,t-1)$.
As we now explain, under certain conditions this DRACKN is guaranteed to be abelian and thus yields ETFs.
In fact, each of these GQs corresponds to a $(v,k,f)$-polyphase BIBD ETF with the special property that $f=k$.

\begin{theorem}
\label{theorem.abelian GQ}
If $\bfPhi(z)$ is a $(v,k,k)$-polyphase BIBD ETF (Definition~\ref{definition.polyphase BIBD ETF}) then
\begin{equation}
\label{equation.GQ in terms of polyphase}
\bfZ
=\left[\begin{array}{c}\bfI_v^{}\otimes\bfone_k^\rmT\\\bfPhi(\bfT)\end{array}\right]
\end{equation}
is the incidence matrix of a $\GQ(k-1,r)$ that contains a spread where $r=\frac{v-1}{k-1}$.
Conversely, if
\begin{equation}
\label{equation.polyphase in terms of GQ}
\bfZ
=\left[\begin{array}{c}\bfI_{st+1}^{}\otimes\bfone_{s+1}^\rmT\\\bfY\end{array}\right]
\end{equation}
is the incidence matrix of a $\GQ(s,t)$ that contains a spread,
and there exists an abelian group $\calG$ of order $s+1$ such that $\bfY$ is a \smash{$\frac{t(st+1)}{s+1}\times(st+1)$} array of $(s+1)\times(s+1)$ of blocks that are either a $\calG$-circulant permutation matrix or the zero matrix,
then there exists an $(st+1,s+1,s+1)$-polyphase BIBD ETF $\bfPhi(z)$ such that $\bfY=\bfPhi(\bfT)$.

In this case, the resulting ETFs $\bfPhi(\gamma)$ have parameters
\begin{equation}
\label{equation.parameters of abelian GQ ETF}
d=\frac{t(st+1)}{s+t},
\quad
n=st+1,
\quad
n-d=\frac{s(st+1)}{s+t}.
\end{equation}
Moreover, when $s+1$ is even, $\gamma$ can be chosen so that $\bfPhi(\gamma)$ is a real ETF.

Also, $\bfPhi^*(z)\bfPhi(z)-t\bfI$ is an abelian $(st+1,s+1,t-1)$-DRACKN.
\end{theorem}

\begin{proof}
($\Rightarrow$)
Let $\bfPhi(z)$ be a $(v,k,k)$-polyphase BIBD ETF over some abelian group $\calG$ of order $k$.
The matrix $\bfPhi(\bfT)$ is obtained by identifying each entry of $\bfPhi(z)$ with a $\calG$-circulant matrix according to the isomorphism~\eqref{equation.definition of x(T)}.
Thus, $\bfPhi(\bfT)$ is a $b\times v$ array of blocks of size $k\times k$, where each block is either zero or a $\calG$-circulant permutation matrix.
In particular, the matrix $\bfZ$ given in~\eqref{equation.GQ in terms of polyphase} is a well-defined incidence matrix.
Moreover, since $\bfX:=\abs{\bfPhi(z)}^2$ is the incidence matrix of a $\BIBD(v,k,1)$,
each row of $\bfPhi(\bfT)$ contains exactly $k$ ones,
while each column of $\bfPhi(\bfT)$ contains exactly $r$ ones.
Letting $s=k-1$ and $t=r$,
this implies the rows and columns of $\bfZ$ contain exactly $s+1$ ones and $t+1$ ones, respectively.
That is, $\bfZ$ satisfies~\eqref{equation.definition of GQ 1}.

We next show $\bfZ$ also satisfies~\eqref{equation.definition of GQ 2}.
To do so, note
\begin{equation}
\label{equation.proof of abelian GQ 1}
\bfZ\bfZ^\rmT
=\left[\begin{array}{c}\bfI_v^{}\otimes\bfone_k^\rmT\\\bfPhi(\bfT)\end{array}\right]\left[\begin{array}{cc}\bfI_v^{}\otimes\bfone_k^{}&\bfPhi^*(\bfT)\end{array}\right]
=\left[\begin{array}{cc}k\bfI_v&(\bfI_v^{}\otimes\bfone_k^\rmT)\bfPhi^*(\bfT)\\\bfPhi(\bfT)(\bfI_v^{}\otimes\bfone_k^{})&\bfPhi(\bfT)\bfPhi^*(\bfT)\end{array}\right].
\end{equation}
Since $\bfX$ is the incidence matrix of a $\BIBD(v,k,1)$, any two of its blocks intersect in at most one vertex.
This means that the off-diagonal entries of $\bfPhi(z)\bfPhi^*(z)$ are either monomials or zero while its diagonal entries are $k$.
Thus, $\bfPhi(\bfT)\bfPhi^*(\bfT)$ is a block matrix whose diagonal $k\times k$ blocks are $k\bfI$ and whose off-diagonal blocks are either permutation matrices or zero.
In particular, the off-diagonal entries of the lower-right term in~\eqref{equation.proof of abelian GQ 1} are $\set{0,1}$-valued.
Meanwhile, the lower-left term in~\eqref{equation.proof of abelian GQ 1} is a $(b\times v)\times(k\times 1)$ block matrix whose $(i,j)$th block is
\begin{equation*}
[\bfPhi(\bfT)(\bfI_v^{}\otimes\bfone_k^{})](i,j)
=\sum_{j'=1}^{v}\bfPhi(\bfT;i,j')(\bfI_v^{}\otimes\bfone_k^{})(j',j)
=\bfPhi(\bfT;i,j)\bfone_k^{},
\end{equation*}
which is either $\bfone_k$ or $\bfzero_k$ depending on whether $\bfPhi(\bfT;i,j)$ is a permutation matrix or is the zero matrix.
This means
\begin{equation}
\label{equation.proof of abelian GQ 2}
\bfPhi(\bfT)(\bfI_v^{}\otimes\bfone_k^{})=\bfX\otimes\bfone_k^{},
\end{equation}
which is $\set{0,1}$-valued.
Thus, $\bfPhi(\bfT)$ satisfies the first half of~\eqref{equation.definition of GQ 2}.
For the second half of~\eqref{equation.definition of GQ 2},
note that since $\bfX$ is a $\BIBD(v,k,1)$, the off-diagonal entries of $\bfPhi^*(z)\bfPhi(z)$ are monomials while its diagonal entries are $r$.
Thus, $\bfPhi^*(\bfT)\bfPhi(\bfT)$ is a block matrix whose diagonal $k\times k$ blocks are $r\bfI$ and whose off-diagonal blocks are permutation matrices.
In particular, the off-diagonal entries of
\begin{equation}
\label{equation.proof of abelian GQ 3}
\bfZ^\rmT\bfZ
=\left[\begin{array}{cc}\bfI_v^{}\otimes\bfone_k^{}&\bfPhi^*(\bfT)\end{array}\right]\left[\begin{array}{c}\bfI_v^{}\otimes\bfone_k^\rmT\\\bfPhi(\bfT)\end{array}\right]
=\bfI_v\otimes\bfJ_k+\bfPhi^*(\bfT)\bfPhi(\bfT)
\end{equation}
are $\set{0,1}$-valued.

Having~\eqref{equation.definition of GQ 1} and~\eqref{equation.definition of GQ 2} we turn to~\eqref{equation.definition of GQ 3}.
Multiplying~\eqref{equation.proof of abelian GQ 3} by $\bfZ$ gives
\begin{equation}
\label{equation.proof of abelian GQ 4}
\bfZ\bfZ^\rmT\bfZ
=\left[\begin{array}{c}\bfI_v^{}\otimes\bfone_k^\rmT\\\bfPhi(\bfT)\end{array}\right]\bigl\{\bfI_v\otimes\bfJ_k+\bfPhi^*(\bfT)\bfPhi(\bfT)\bigr\}
=\left[\begin{array}{r}
(\bfI_v^{}\otimes\bfone_k^\rmT)\bigl[\bfI_v\otimes\bfJ_k+\bfPhi^*(\bfT)\bfPhi(\bfT)\bigr]\smallskip\\
\bfPhi(\bfT)\bigl[\bfI_v\otimes\bfJ_k+\bfPhi^*(\bfT)\bfPhi(\bfT)\bigr]
\end{array}\right].
\end{equation}
Both the top and bottom terms of~\eqref{equation.proof of abelian GQ 4} can be simplified with~\eqref{equation.proof of abelian GQ 2}.
Specifically,
\begin{align*}
(\bfI_v^{}\otimes\bfone_k^\rmT)\bigl[\bfI_v\otimes\bfJ_k+\bfPhi^*(\bfT)\bfPhi(\bfT)\bigr]
&=k(\bfI_v^{}\otimes\bfone_k^\rmT)+\bigl[\bfPhi(\bfT)(\bfI_v^{}\otimes\bfone_k^{})\bigr]^\rmT\bfPhi(\bfT)\\
&=k(\bfI_v^{}\otimes\bfone_k^\rmT)+(\bfX\otimes\bfone_k^{})^\rmT\bfPhi(\bfT).
\end{align*}
To continue simplifying this term, note the $(j,j')$th block of size $1\times k$ of $(\bfX\otimes\bfone_k^{})^\rmT\bfPhi(\bfT)$ is
\begin{equation*}
[(\bfX\otimes\bfone_k^{})^\rmT\bfPhi(\bfT)](j,j')
=\sum_{i=1}^{b}\bfX(i,j)\bfone_k^{\rmT}\bfPhi(\bfT;i,j')
=\sum_{i=1}^{b}\bfX(i,j)\bfX(i,j')\bfone_k^{\rmT}
=(\bfX^\rmT\bfX)(j,j')\bfone_k^{\rmT}.
\end{equation*}
Thus, $(\bfX\otimes\bfone_k^{})^\rmT\bfPhi(\bfT)
=(\bfX^\rmT\bfX)\otimes\bfone_k^\rmT
=[(r-1)\bfI_v+\bfJ_v]\otimes\bfone_k^\rmT
=(r-1)\bfI_v\otimes\bfone_k^\rmT+\bfJ_{v\times{vk}}$
and so
\begin{equation}
\label{equation.proof of abelian GQ 5}
(\bfI_v^{}\otimes\bfone_k^\rmT)\bigl[\bfI_v\otimes\bfJ_k+\bfPhi^*(\bfT)\bfPhi(\bfT)\bigr]
=(r+k-1)(\bfI_v^{}\otimes\bfone_k^\rmT)+\bfJ_{v\times{vk}}.
\end{equation}
To simplify the bottom term of~\eqref{equation.proof of abelian GQ 4}, note that since $\bfPhi(z)$ is a $(v,k,k)$-polyphase BIBD ETF,
Theorem~\ref{theorem.polyphase BIBD ETF}.(ii) gives
$\bfPhi(z)\bfPhi^*(z)\bfPhi(z)=(r+k-1)\bfPhi(z)+\bfone(z)(\bfJ-\bfX)$.
Under the isomorphism~\eqref{equation.definition of x(T)}, $\bfone(z)$ becomes $\bfone(\bfT)=\bfJ_k$ and so this equation becomes
\begin{equation*}
\bfPhi(\bfT)\bfPhi^*(\bfT)\bfPhi(\bfT)
=(r+k-1)\bfPhi(\bfT)+(\bfJ_{b\times v}-\bfX)\otimes\bfJ_k.
\end{equation*}
Using this along with~\eqref{equation.proof of abelian GQ 2} we can write the bottom term in~\eqref{equation.proof of abelian GQ 4} as
\begin{align}
\nonumber
\bfPhi(\bfT)\bigl[\bfI_v\otimes\bfJ_k+\bfPhi^*(\bfT)\bfPhi(\bfT)\bigr]
&=\bfPhi(\bfT)(\bfI_v^{}\times\bfone_k^{})(\bfI_v^{}\times\bfone_k^{\rmT})+\bfPhi(\bfT)\bfPhi^*(\bfT)\bfPhi(\bfT)\\
\nonumber
&=\bfX\otimes\bfJ_k+(r+k-1)\bfPhi(\bfT)+(\bfJ_{b\times v}-\bfX)\otimes\bfJ_k\\
\label{equation.proof of abelian GQ 6}
&=(r+k-1)\bfPhi(\bfT)+\bfJ_{bk\times vk}.
\end{align}
Combining~\eqref{equation.proof of abelian GQ 4}, \eqref{equation.proof of abelian GQ 5}, \eqref{equation.proof of abelian GQ 6} and the fact that $r+k-1=s+t$, we see that $\bfZ$ satisfies~\eqref{equation.definition of GQ 3}:
\begin{equation*}
\bfZ\bfZ^\rmT\bfZ
=\left[\begin{array}{l}
(r+k-1)(\bfI_v^{}\otimes\bfone_k^\rmT)+\bfJ_{v\times{vk}}\\
(r+k-1)\bfPhi(\bfT)+\bfJ_{bk\times vk}
\end{array}\right]
=(r+k-1)\left[\begin{array}{c}\bfI_v^{}\otimes\bfone_k^\rmT\\\bfPhi(\bfT)\end{array}\right]+\bfJ
=(r+k-1)\bfZ+\bfJ.
\end{equation*}

($\Leftarrow$)
Now assume we have a $\GQ(s,t)$ whose incidence matrix $\bfZ$ is of form~\eqref{equation.polyphase in terms of GQ} where $\bfY$ is a block matrix whose blocks are either zero or permutation matrices that are circulant with respect to some given abelian group $\calG$ of order $s+1$.
Since $\bfZ$ is $(t+1)(st+1)\times(s+1)(st+1)$, $\bfY$ is $t(st+1)\times(s+1)(st+1)$.
Thus, $\bfY$ is a $b\times v$ array of $k\times k$ blocks where $k=s+1$, $v=st+1$, $r=t$ and
\smash{$b=\frac vkr=\frac{t(st+1)}{s+1}$}.
Let $\bfX$ and $\bfPhi(z)$ be the $b\times v$ incidence matrix and polyphase matrix corresponding to $\bfY$, respectively:
\begin{equation*}
\bfX(i,j)=\left\{\begin{array}{cl}1,&\bfY(i,j)\neq\bfzero,\\0,&\bfY(i,j)=\bfzero,\end{array}\right.
\qquad
\bfPhi(z;i,j)=\left\{\begin{array}{cl}z^g,&\bfY(i,j)=\bfT^g,\\0,&\bfY(i,j)=\bfzero.\end{array}\right.
\end{equation*}
Clearly, the entries of $\bfPhi(z)$ are either monomials or zero and $\abs{\bfPhi(z)}^2=\bfX$.
To show $\bfPhi(z)$ is a $(st+1,s+1,s+1)=(v,k,1)$-polyphase BIBD ETF,
we show that $\bfX$ is the incidence matrix of a $\BIBD(v,k,1)$ and that $\bfPhi(z)$ satisfies Theorem~\ref{theorem.polyphase BIBD ETF}(ii).

To show that $\bfX$ is the incidence matrix of a $\BIBD(v,k,1)$,
note that since any row of $\bfY$ contains exactly $k=s+1$ ones,
each block-row of $\bfY$ contains $k$ permutation matrices,
meaning each row of $\bfX$ contains $k$ ones.
As such, what remains is to show that any two distinct columns of $\bfX$ have exactly one row index of common support.
For any $j=1,\dotsc,st+1$ let $\bfY_j$ denote the $j$th submatrix of $\bfY$ of size $t(st+1)\times(s+1)$.
That is, $\bfY=[\bfY_1 \dotsb \bfY_{st+1}]$.
Under this notation, $\bfZ^\rmT\bfZ$ is an $(st+1)\times(st+1)$ array of blocks of size $(s+1)\times(s+1)$ whose $(j,j')$th block is
\begin{equation}
\label{equation.proof of abelian GQ 7}
(\bfZ^\rmT\bfZ)(j,j')
=\left[\begin{array}{c}\bfdelta_j^{}\otimes\bfone_k^\rmT\\\bfY_j\end{array}\right]^\rmT\left[\begin{array}{c}\bfdelta_{j'}^{}\otimes\bfone_k^\rmT\\\bfY_{j'}\end{array}\right]
=\left\{\begin{array}{cl}t\bfI+\bfJ,&j=j',\\\bfY_j^\rmT\bfY_{j'}^{},&j\neq j'.\end{array}\right.
\end{equation}
By~\eqref{equation.definition of GQ 2}, this implies $\bfY_j^\rmT\bfY_{j'}^{}$ is $\set{0,1}$-valued for any $j\neq j'$.

We claim that \smash{$\bfY_j^\rmT\bfY_{j'}^{}$} is a permutation matrix for all $j\neq j'$.
To prove this claim, we borrow a relevant idea from~\cite{Brouwer84}.
In particular, recall from~\eqref{equation.square of GQ Gram} that $\bfZ^\rmT\bfZ-(t+1)\bfI$ is an SRG~\eqref{equation.definition of SRG} with parameters~\eqref{equation.DRACKN SRG parameters}.
In particular, each of the $(s+1)(st+1)$ vertices of this graph has $s(t+1)$ neighbors,
and any two adjacent vertices have $s-1$ neighbors in common.
Further note that~\eqref{equation.proof of abelian GQ 7} implies this SRG is partitioned into $st+1$ cliques of size $s+1$ corresponding to the diagonal blocks of $\bfZ^\rmT\bfZ$.
Together, these facts imply that if two vertices are in a common clique, then their common neighbors are precisely the remaining $s-1$ vertices in that clique.
That is, any vertex can have at most one neighbor from each of the $st$ cliques it does not belong to.
At the same time, it has $s(t+1)$ neighbors overall, including $s$ from its own clique, meaning each vertex has exactly one neighbor from each of the $st$ cliques to which it does not belong.
In particular, for any $j\neq j'$, each row and column of \smash{$\bfY_j^\rmT\bfY_{j'}^{}$} contains exactly one $1$, meaning it is a permutation matrix, as claimed.

Having this claim, note $\bfY_j$ is a vertical concatenation of the matrices $\set{\bfY(i,j)}_{j=1}^{b}$ which are either permutation matrices or zero.
Moreover, for any $j\neq j'$ the permutation matrix \smash{$\bfY_j^\rmT\bfY_{j'}^{}$} is a sum of products of such matrices
\begin{equation*}
\bfY_j^\rmT\bfY_{j'}
=\sum_{i=1}^{b}\bfY(i,j)\bfY(i,j').
\end{equation*}
Since a permutation matrix cannot be written as a sum of two or more permutation matrices,
this means there is exactly one value of $i=i(j,j')$ such that $\bfY(i,j)\neq\bfzero\neq\bfY(i,j')$.
Thus,
any two distinct columns of $\bfX$ have exactly one index of common support,
meaning it is indeed the incidence matrix of a $\BIBD(v,k,1)$.

To show that $\bfPhi(z)$ satisfies the condition of Theorem~\ref{theorem.polyphase BIBD ETF}(ii),
note that since $v=st+1$, $k=s+1$ and $\bfY=\bfPhi(\bfT)$, our matrix~\eqref{equation.polyphase in terms of GQ} is of form~\eqref{equation.GQ in terms of polyphase} where $\bfX=\abs{\bfPhi(z)}^2$ is the incidence matrix of a $\BIBD(v,k,1)$.
This means that certain facts from the proof of the converse direction,
such as~\eqref{equation.proof of abelian GQ 2} and~\eqref{equation.proof of abelian GQ 4} are also valid here.
In particular, using~\eqref{equation.proof of abelian GQ 4} we can rewrite our assumption~\eqref{equation.definition of GQ 3} as
\begin{equation*}
\left[\begin{array}{r}
(\bfI_v^{}\otimes\bfone_k^\rmT)\bigl[\bfI_v\otimes\bfJ_k+\bfPhi^*(\bfT)\bfPhi(\bfT)\bigr]\smallskip\\
\bfPhi(\bfT)\bigl[\bfI_v\otimes\bfJ_k+\bfPhi^*(\bfT)\bfPhi(\bfT)\bigr]
\end{array}\right]
=\left[\begin{array}{l}(r+k-1)(\bfI_{v}^{}\otimes\bfone_{k}^\rmT)+\bfJ_{v\times vk}\\(r+k-1)\bfPhi(\bfT)+\bfJ_{bk\times vk}\end{array}\right].
\end{equation*}
By~\eqref{equation.proof of abelian GQ 2}, the bottom half of this equation becomes
\begin{equation*}
\bfX\otimes\bfJ_k+\bfPhi(\bfT)\bfPhi^*(\bfT)\bfPhi(\bfT)
=(r+k-1)\bfPhi(\bfT)+\bfJ_{bk\times vk}
=(r+k-1)\bfPhi(\bfT)+\bfJ_{b\times v}\otimes\bfJ_{k},
\end{equation*}
that is,
$\bfPhi(\bfT)\bfPhi^*(\bfT)\bfPhi(\bfT)=(r+k-1)\bfPhi(\bfT)+(\bfJ-\bfX)\otimes\bfJ_k$.
Under the isomorphism~\eqref{equation.definition of x(T)}, $\bfJ_k$ becomes $\bfone(z)$ and this equation becomes $\bfPhi(z)\bfPhi^*(z)\bfPhi(z)=(r+k-1)\bfPhi(z)+\bfone(z)(\bfJ-\bfX)$,
namely the condition of Theorem~\ref{theorem.polyphase BIBD ETF} in the special case where $f=s+1=k$.

For~\eqref{equation.parameters of abelian GQ ETF}, note that in this case,
Theorem~\ref{theorem.polyphase BIBD ETF} gives that $\bfPhi(z)$ is a polyphase BIBD ETF with parameters $(v,k,k)=(st+1,s+1,s+1)$,
meaning $\bfPhi(\gamma)$ is a phased $\BIBD(st+1,s+1,1)$ ETF for any nontrivial character $\gamma$.
By Theorem~\ref{theorem.phased BIBD ETFs}, this means the $st+1$ columns of $\bfPhi(\gamma)$ form an ETF for their span which has dimension
\begin{equation*}
d
=\frac{vr}{r+k-1}
=\frac{t(st+1)}{s+t},
\end{equation*}
which in turn implies $n-d=\frac{s(st+1)}{s+t}$.
Also, if the order $s+1$ of $\calG$ is even, then it has a real-valued character $\gamma$,
meaning in that case $\bfPhi(\gamma)$ is real.

Finally, since $\bfPhi(z)$ is a $(st+1,s+1,s+1)$-polyphase BIBD ETF whose underlying BIBD has $k=s+1$ and $r=t$,
Theorem~\ref{theorem.polyphase BIBD ETF} also immediately gives that $\bfPhi^*(z)\bfPhi(z)-t\bfI$ is an abelian $(st+1,s+1,c)$-DRACKN where $c=\frac{k(r-1)}{f}=\frac{(s+1)(t-1)}{s+1}=t-1$.
\end{proof}

In light of the previous result, we make the following definition:

\begin{definition}
\label{definition.abelian GQ}
We say a generalized quadrangle $\GQ(s,t)$ is \textit{abelian} if it has an incidence matrix of the form~\eqref{equation.GQ in terms of polyphase} where $\bfPhi(z)$ is an $(st+1,s+1,s+1)$-polyphase BIBD ETF,
cf.\ Definition~\ref{definition.polyphase BIBD ETF}.
\end{definition}

Note that in order for a $\GQ(s,t)$ to be abelian,
applying Fisher's inequality to the underlying $\BIBD(st+1,s+1,1)$ gives $s=k-1<r=t$.
This means that of the parameters~\eqref{equation.parameters of known GQ} of known constructions of $\GQ(s,t)$ we should only consider $(s,t)$ of the form $(1,r)$ for some $r\geq2$ or $(q,q^2)$, $(q^2,q^3)$, $(q-1,q+1)$ for some prime power $q$.
Note Theorem~\ref{theorem.abelian GQ} also implies that both $s+1$ and $s+t$ necessarily divide $t(st+1)$.
This is more restrictive than having $s+t$ divide $st(s+1)(t+1)$, which is a known necessary condition on the parameters of any $\GQ(s,t)$.
For example, though there exists a $\GQ(q^2,q^3)$ for a prime power $q$, this GQ cannot be abelian since
\begin{equation*}
\frac{t(st+1)}{s+1}
=\frac{q^3(q^5+1)}{q^2+1}
=q^6-q^4+q^2+q-1-\frac{q-1}{q^2+1}
\end{equation*}
is not an integer for any $q\geq 2$.
Altogether, we see that the only known constructions of $\GQ(s,t)$ that might be abelian are those whose $(s,t)$ parameters are of the form
\begin{equation}
\label{equation.parameters of known abelian GQ}
(1,r),
\quad
(q,q^2),
\quad
(q-1,q+1),
\end{equation}
where $r\geq2$ and $q$ is a prime power.
In the next section, we produce explicit constructions of abelian $\GQ(s,t)$ for all three of these types of parameters.
To be clear, other abelian GQs may exist.
For example, any $\GQ(s,t)$ with $s\geq 3$ and $t=s^2-s-1$ has $s<t$ and
\begin{equation*}
\frac{t(st+1)}{s+1}
=(s-1)^2(s^2-s-1),
\quad
\frac{t(st+1)}{s+t}
=(s-1)(s^2-s-1).
\end{equation*}
However, with the exception of $(s,t)=(3,5)$ (which is constructed by letting $(s,t)=(q-1,q+1)$ when $q=4$), no constructions of $\GQ(s,t)$ with $t=s^2-s-1$ are known.
To the best of our knowledge, the existence of GQs with these parameters is an open question.
For example, the existence of a $\GQ(4,11)$ is open~\cite{Payne07}.
We note that by Theorem~\ref{theorem.abelian GQ}, if a $\GQ(s,s^2-s-1)$ did exist the resulting ETF would have
\begin{equation*}
n=st+1=(s-1)^2(s+1),
\quad
d=(s-1)(s^2-s-1),
\quad
n-d =s(s-1).
\end{equation*}
That is, a Naimark complement of an ETF arising from a $\GQ(s,s^2-s-1)$ has the same parameters as a Steiner ETF arising from an affine plane of order $s-1$.
This suggests we let $q=s-1$ be a prime power and seek an abelian $\GQ(q+1,q^2+q-1)$.
In such a GQ, any block contains $q+1$ vertices.
As such, an abelian $\GQ(q+1,q^2+q-1)$ might lie as a subincidence structure of a projective geometry over the field $\bbF_q$ of $q$ elements.
We leave a deeper investigation of this question for future work.

The above line of reasoning begs the following question:
by Theorem~\ref{theorem.abelian GQ}, the existence of an abelian GQ implies the existence of a phased BIBD ETF;
is the converse ever true?
That is, does the existence of a phased BIBD ETF imply the existence of an abelian GQ?
By borrowing the proof technique of~Theorem 5.1 of~\cite{CoutinkhoGSZ16}, we see the answer to this question is yes:

\begin{theorem}
\label{theorem.ETF implies GQ}
Suppose $\bfPhi$ is a phased $\BIBD(v,k,1)$ ETF whose nonzero entries are $p$th roots of unity where $p$ is prime.
Then $p$ divides $k$ and there exists a $(v,k,p)$-polyphase BIBD ETF $\bfPhi(z)$ over $\bbZ_p$ such that $\bfPhi=\bfPhi(\gamma)$ for some character $\gamma$.

In particular, if $k=p$ then~\eqref{equation.GQ in terms of polyphase} is the incidence matrix of a $\GQ(k-1,r)$ where $r=\frac{v-1}{k-1}$.
\end{theorem}

\begin{proof}
Let $\omega=\exp(\frac{2\pi\rmi}{p})$ and let $\bfPhi(z)$ be the $b\times v$ polyphase matrix whose $(i,j)$th entry is
\begin{equation}
\label{equation.proof of ETF implies GQ 0}
\bfPhi(z;i,j)
=\left\{\begin{array}{cl}z^l,&\bfPhi(i,j)=\omega^l,\\0,&\bfPhi(i,j)=0,\end{array}\right.
\end{equation}
regarded as a polynomial over $\bbZ_p$, that is, as a member of $\bbZ_p[z]/\gen{z^p-1}$.
Note $\abs{\bfPhi(z)}^2=\bfX=\abs{\bfPhi}^2$ is the incidence matrix of a $\BIBD(v,k,1)$.
For any $(i,j)$ such that $\bfPhi(i,j)=0$, \eqref{equation.phased BIBD ETF condition} gives
\begin{equation}
\label{equation.proof of ETF implies GQ 1}
\sum_{i'=1}^{b}\sum_{j'=1}^{v}\bfPhi(i,j')\overline{\bfPhi(i',j')}\bfPhi(i',j)=0.
\end{equation}
As noted in our discussion following the proof of~Theorem~\ref{theorem.phased BIBD ETFs},
only $k$ of these summands are nonzero.
Here, these nonzero summands are all $p$th roots of unity, being the product of three $p$th roots of unity.
Since $p$ is prime, the $p$th cyclotomic polynomial is $\sum_{l=0}^{p-1}z^l$.
This implies that if $\set{c_l}_{l=0}^{p-1}$ are rational numbers such that $\sum_{l=0}^{p-1}c_l\omega^l=0$ then $c_1=\dotsb=c_p$.
In particular, the $k$ nonzero summands of~\eqref{equation.proof of ETF implies GQ 1} consist of $\frac kp$ copies of each $p$th root of unity, that is,
\begin{equation*}
\#\set{(i',j'): \omega^l = \bfPhi(i,j')\overline{\bfPhi(i',j')}\bfPhi(i',j)}=\tfrac kp.
\end{equation*}
By~\eqref{equation.proof of ETF implies GQ 0},
this implies $\#\set{(i',j'): z^l = \bfPhi(z;i,j')\overline{\bfPhi(z;i',j')}\bfPhi(z;i',j)}=\tfrac kp$ is independent of $l$, and so Theorem~\ref{theorem.polyphase BIBD ETF}(iv) gives $\bfPhi(z)$ is a $(v,k,p)$-polyphase BIBD ETF over $\bbZ_p$.
In particular, if $p=k$ then Theorem~\ref{theorem.abelian GQ} implies~\eqref{equation.GQ in terms of polyphase} is the incidence matrix of a $\GQ(k-1,r)$ where $r=\frac{v-1}{k-1}$.
\end{proof}

Combining Theorems~\ref{theorem.abelian GQ} and~\ref{theorem.ETF implies GQ},
we see that if $s+1$ is prime then an abelian $\GQ(s,t)$ exists if and only if there exists a phased $\BIBD(st+1,s+1,1)$ ETF whose nonzero entries are $(s+1)$th roots of unity.
In particular, returning to the line of reasoning given immediately before Theorem~\ref{theorem.ETF implies GQ}, we see that when $s+1$ is prime, an abelian $\GQ(s,s^2-s-1)$ exists if and only if there exists a phased $\BIBD((s-1)^2(s+1),s+1,1)$ ETF whose nonzero entries are $(s+1)$th roots of unity.
This is intriguing since, as mentioned earlier, the existence of $\GQ(s,s^2-s-1)$ seems to be unresolved for all $s>3$~\cite{Payne07}, while $\BIBD((s-1)^2(s+1),s+1,1)$ are known to exist whenever $3\leq s\leq 7$~\cite{MathonR07}.
For example, perhaps a $\GQ(4,11)$ can be obtained by phasing a known example of a $\BIBD(45,5,1)$ with fifth roots of unity.

\section{Explicit constructions of abelian generalized quadrangles}

In this section, we construct abelian generalized quadrangles with parameters~\eqref{equation.parameters of known abelian GQ} for any $r\geq2$ and prime power $q$, cf.\ Definition~\ref{definition.abelian GQ}.
By Theorem~\ref{theorem.abelian GQ}, each of these produces an ETF.

The first of these constructions is an abelian $\GQ(1,r)$, which is a trivial type of GQ known as a \textit{dual grid}~\cite{PayneT09}.
They correspond to trivial ETFs, namely regular simplices.
To be precise, for any $r\geq 2$, an abelian $\GQ(1,r)$ can be obtained by letting $v=r+1$ and applying Theorem~\ref{theorem.ETF implies GQ} to the phased $\BIBD(v,2,1)$ ETF $\bfPhi$ given in Example~\ref{example.regular simplices as phased BIBD ETFs} whose nonzero entries lie in $\set{\pm1}$, namely $p$th roots of unity where $p=2$ is prime.
For example, for the phased BIBD Mercedes-Benz ETF given in~\eqref{equation.3x3 phased BIBD ETF},
the method of Theorem~\ref{theorem.ETF implies GQ} gives the following $(3,2,2)$-polyphase BIBD ETF (whose entries are polynomials in $\bbC[z]/\gen{z^2-1}$) as well as the incidence matrix $\bfZ$ of a $\GQ(1,2)$:
\begin{equation*}
\bfPhi(z)
=\left[\begin{array}{rrr}1&z&0\\1&0&z\\0&1&z\end{array}\right],
\qquad
\bfZ
=\left[\begin{array}{c}\bfI_3^{}\otimes\bfone_2^\rmT\\\bfPhi(\bfT)\end{array}\right]
=\left[\begin{array}{cc|cc|cc}
1&1&0&0&0&0\\
0&0&1&1&0&0\\
0&0&0&0&1&1\\
\hline
1&0&0&1&0&0\\
0&1&1&0&0&0\\
\hline
1&0&0&0&0&1\\
0&1&0&0&1&0\\
\hline
0&0&1&0&0&1\\
0&0&0&1&1&0
\end{array}\right].
\end{equation*}

Constructions of abelian $\GQ(q-1,q+1)$ and $\GQ(q,q^2)$ are far less obvious.
In particular, we were not able to deduce an abelian structure on known constructions of $\GQ(q-1,q+1)$~\cite{Payne07,PayneT09}.
Instead, we construct an abelian $\GQ(q-1,q+1)$ from scratch.

\subsection{Constructing an abelian $\GQ(q-1,q+1)$}

For any prime power $q$, let $\calG$ be the (abelian) additive group of the finite field $\bbF_q$.
Note the $(q^2,q,q)$-polyphase BIBD ETF produced by Theorem~\ref{theorem.abelian GQ} from any abelian $\GQ(q-1,q+1)$ over this group has parameters
\begin{equation}
\label{equation.ETF parameters of abelian GQ(q-1,q+1)}
d=\binom{q+1}{2},
\quad
n=q^2
\quad
n-d=\binom{q}{2}.
\end{equation}
In particular, a Naimark complement of this ETF has the same size as a Steiner ETF arising from a $\BIBD(q,2,1)$.
Such Steiner ETFs are well understood.
Also, the underlying BIBDs have parameters $(v,k,1)=(q^2,q,1)$ and can be constructed, for example, as affine planes over $\bbF_q$.
Together, these facts led us to the $(9,3,3)$-polyphase BIBD ETF over $\bbZ_q$ given in Example~\ref{example.12x9 phased BIBD ETF}, as well as to the following generalization of it over the additive group of $\bbF_q$:

\begin{theorem}
\label{theorem.(q^2,q,q) polyphase BIBD ETFs}
For any prime power $q$,
let $\bfPhi(z)$ be a $[(q+1)\times q]\times(q\times q)$ block matrix whose entries are polynomials over the additive group of $\bbF_q$.
Specifically, for any $x,y,j\in\bbF_q$ and any $i\in\bbF_q\cup\set{\infty}$,
let $\bfPhi(z)$ be the matrix whose $(x,y)$th entry of its $(i,j)$th block is
\begin{equation}
\label{equation.definition of polyphase affine}
\bfPhi(z;i,j;x,y):=\left\{\begin{array}{cl}z^{j(x+y)},&x-y=ij,\ i\neq\infty,\\1,&x=j,\ i=\infty,\\0,&\text{else}.\end{array}\right.
\end{equation}
Then $\bfPhi(z)$ is a $(q^2,q,q)$-polyphase BIBD ETF.  As such:
\begin{enumerate}
\alphi
\item
For any nontrivial character $\gamma$,
$\bfPhi(\gamma)$ is a phased BIBD ETF with parameters~\eqref{equation.ETF parameters of abelian GQ(q-1,q+1)}.
\item
When $q=2^j$, $\gamma$ can be chosen so that $\bfPhi(\gamma)$ is a real ETF.
\item
$\bfPhi(z)^*\bfPhi(z)-(q+1)\bfI$ is an abelian $(q^2,q,q)$-DRACKN.
\item
The matrix \eqref{equation.GQ in terms of polyphase} is the incidence matrix of an abelian $\GQ(q-1,q+1)$.
\end{enumerate}
\end{theorem}

\begin{proof}
Let $\bfX=\abs{\bfPhi(z)}^2$.
We claim $\bfX$ is the incidence matrix of a $\BIBD(q^2,q,1)$,
namely an affine plane of order $q$.
(A remark on notation: the index $i$ here corresponds to the ``slope" of a parallel class of affine lines,
with ``$\infty$" corresponding to vertical lines.)
Indeed, for any $(i,x)$ there are exactly $q$ choices of $(j,y)$ such that $\bfX(i,j;x,y)=1$:
when $i\neq\infty$, $y=x-ij$ where $j$ is arbitrary;
when $i=\infty$, $j=x$ and $y$ is arbitrary.
Moreover, if $(j,y)\neq(j',y')$ then there exists exactly one choice of $(i,x)$ such that $\bfX(i,j;x,y)=1=\bfX(i,j';x,y')$:
when $j\neq j'$ we have $x=y+ij=y'+ij'$ where $i=-(y-y')(j-j')^{-1}$;
when $j=j'$ and $y\neq y'$ we have $x=j$ and $i=\infty$.

Having that $\bfX$ is a $\BIBD(q^2,q,1)$, we use Theorem~\ref{theorem.polyphase BIBD ETF}(ii) to show $\bfPhi(z)$ is a $(q^2,q,q)$-polyphase BIBD ETF.
For any $j,j',y,y'\in\bbF_q$, the $(y,y')$th entry of the $(j,j')$th block of $[\bfPhi(z)]^*\bfPhi(z)$ is
\begin{equation}
\label{equation.proof of polyphase affine 1}
[\bfPhi^*(z)\bfPhi(z)](j,j';y,y')
=\sum_{i\in\bbF_q\cup\set{\infty}}\sum_{x\in\bbF_q}\overline{\bfPhi(z;i,j;x,y)}\bfPhi(z;i,j';x,y').
\end{equation}
By \eqref{equation.definition of polyphase affine}, the ``$i=\infty$" part of this sum is
\begin{equation}
\label{equation.proof of polyphase affine 2}
\sum_{x\in\bbF_q}\overline{\bfPhi(z;\infty,j;x,y)}\bfPhi(z;\infty,j';x,y')
=\sum_{x\in\bbF_q}\left\{\begin{array}{cl}1,&x=j=j'\\0,&\text{else}\end{array}\right\}
=\left\{\begin{array}{cl}1,&j=j',\\0,&\text{else}.\end{array}\right.
\end{equation}
Meanwhile, the remaining part of \eqref{equation.proof of polyphase affine 1} is
\begin{equation*}
\sum_{i\in\bbF_q}\sum_{x\in\bbF_q}\overline{\bfPhi(z;i,j;x,y)}\bfPhi(z;i,j';x,y')
=\sum_{i\in\bbF_q}\sum_{x\in\bbF_q}\left\{\begin{array}{cl}z^{-j(x+y)+j'(x+y')},& x-y=ij, x-y'=ij'\\0,&\text{else}\end{array}\right\}.
\end{equation*}
The summands above are thus nonzero only when $x=y+ij$ and $x=y'+ij'$.
When $y+ij\neq y'+ij'$, there is no such $x$.
When $y+ij=y'+ij'$, there is exactly one such $x$, and in this case the exponent of the summand simplifies to
\begin{equation*}
-j(x+y)+j'(x+y')
=-j(y'+ij'+y)+j'(y+ij+y')
=-(j-j')(y+y').
\end{equation*}
That is, the ``$i\neq\infty$" part of \eqref{equation.proof of polyphase affine 1} is
\begin{equation*}
\sum_{i\in\bbF_q}\sum_{x\in\bbF_q}\overline{\bfPhi(z;i,j;x,y)}\bfPhi(z;i,j';x,y')
=\sum_{i\in\bbF_q}\left\{\begin{array}{cl}z^{-(j-j')(y+y')},& y+ij=y'+ij'\\0,&\text{else}\end{array}\right\}.
\end{equation*}
To simplify this even further,
note $y+ij=y'+ij'$ precisely when $y-y'=i(j'-j)$.
When $j'\neq j$, there is exactly one such value of $i$, namely $i=-(y-y')(j-j')^{-1}$.
Meanwhile, when $j'=j$ but $y\neq y'$, no $i\in\bbF_q$ gives $y+ij=y'+ij'$;
when $j=j'$ and $y=y'$, all $i\in\bbF_q$ do.
Thus,
\begin{equation}
\label{equation.proof of polyphase affine 3}
\sum_{i\in\bbF_q}\sum_{x\in\bbF_q}\overline{\bfPhi(z;i,j;x,y)}\bfPhi(z;i,j';x,y')
=\left\{\begin{array}{cl}q,&j=j',y=y',\\0&j=j',y\neq y',\\z^{-(j-j')(y+y')},&j\neq j'.\end{array}\right.
\end{equation}
Summing~\eqref{equation.proof of polyphase affine 2} and~\eqref{equation.proof of polyphase affine 3} gives the following expression for~\eqref{equation.proof of polyphase affine 1}:
\begin{equation*}
[\bfPhi^*(z)\bfPhi(z)](j,j';y,y')
=\left\{\begin{array}{cl}q+1,&j=j',y=y',\\1&j=j',y\neq y',\\z^{-(j-j')(y+y')},&j\neq j'.\end{array}\right.
\end{equation*}
This can be nicely summarized as $\bfPhi^*(z)\bfPhi(z)=q\bfI+\bfPsi(z)$ where $\bfPsi(z;j,j',y,y')=z^{-(j-j')(y+y')}$.
As such, $\bfPhi(z)\bfPhi^*(z)\bfPhi(z)=q\bfPhi(z)+\bfPhi(z)\bfPsi(z)$.
Here, for any $i,j,y,z\in\bbF_q$,
\begin{align}
[\bfPhi(z)\bfPsi(z)](i,j;x,y)
\nonumber&=\sum_{j'\in\bbF_q}\sum_{y'\in\bbF_q}\bfPhi(z;i,j';x,y')\bfPsi(z;j',j;y',y)\\
\nonumber&=\sum_{j'\in\bbF_q}\sum_{y'\in\bbF_q}\left\{\begin{array}{cl}z^{j'(x+y')}z^{(j-j')(y+y')},&x-y'=ij'\\0,&\text{else}\end{array}\right\}\\
\nonumber
\nonumber
\nonumber&=z^{j(x+y)}\sum_{j'\in\bbF_q}z^{j'(x-y-ij)}\\
&=\left\{\begin{array}{cl}qz^{j(x+y)},&x-y=ij,\\\bfone(z),&\text{else},\end{array}\right.
\label{equation.proof of polyphase affine 4}
\end{align}
where $\bfone(z)=\sum_{l\in\bbF_q}z^l$ is the geometric sum over $\bbF_q$.
Meanwhile, for $i=\infty$ and $j,x,y\in\bbF_q$,
\begin{align}
[\bfPhi(z)\bfPsi(z)](\infty,j;x,y)
\nonumber&=\sum_{j'\in\bbF_q}\sum_{y'\in\bbF_q}\bfPhi(z;\infty,j';x,y')\bfPsi(z;j',j;y',y)\\
\nonumber&=\sum_{j'\in\bbF_q}\sum_{y'\in\bbF_q}\left\{\begin{array}{cl}z^{(j-j')(y+y')},&x=j'\\0,&\text{else}\end{array}\right\}\\
\nonumber
\nonumber&=z^{(j-x)y}\sum_{y'\in\bbF_q}z^{(j-x)y'}\\
&=\left\{\begin{array}{cl}q,&x=j,\\\bfone(z),&\text{else}.\end{array}\right.
\label{equation.proof of polyphase affine 5}
\end{align}
Comparing~\eqref{equation.proof of polyphase affine 4} and~\eqref{equation.proof of polyphase affine 5} against \eqref{equation.definition of polyphase affine}, we thus have $\bfPhi(z)\bfPsi(z)=q\bfPhi(z)+\bfone(z)(\bfJ-\bfX)$ and so
\begin{equation*}
\bfPhi(z)\bfPhi^*(z)\bfPhi(z)
=q\bfPhi(z)+\bfPhi(z)\bfPsi(z)
=2q\bfPhi(z)+\bfone(z)(\bfJ-\bfX).
\end{equation*}
By Theorem~\ref{theorem.polyphase BIBD ETF}(ii), this means $\bfPhi(z)$ is a $(q^2,q,q)$-polyphase BIBD ETF.
Having this, (a), (b), (c) and (d) follow quickly from Theorems~\ref{theorem.polyphase BIBD ETF} and~\ref{theorem.abelian GQ}.
\end{proof}

To put Theorem~\ref{theorem.(q^2,q,q) polyphase BIBD ETFs} into context,
ETFs with parameters~\eqref{equation.ETF parameters of abelian GQ(q-1,q+1)},
abelian $(q^2,q,q)$-DRACKNs,
and $\GQ(q-1,q+1)$ were already known to exist~\cite{BodmannPT09,BodmannE10,FickusMT12,CoutinkhoGSZ16,PayneT09}.
The novelty here is that we now know there are $\GQ(q-1,q+1)$ which are abelian,
and so there is an ETF with parameters~\eqref{equation.ETF parameters of abelian GQ(q-1,q+1)} that arises as the columns of a phased $\BIBD(q^2,q,1)$.
Moreover, the construction technique itself seems to be new, being quite different from all other known proofs of the existence of a $\GQ(q-1,q+1)$~\cite{Payne07,PayneT09}.

This construction accounts for half of the known phased BIBD ETFs given in Table~\ref{table.phased BIBD ETFs}, namely those with $u=\frac12(k-1)(k-2)$, cf.\ Theorem~\ref{theorem.phased BIBD necessary conditions}(d).
As we now discuss, the other half arise from abelian $\GQ(q,q^2)$.

\subsection{Constructing an abelian $\GQ(q,q^2)$}

For any prime power $q$,
there are several known constructions of $\GQ(q,q^2)$~\cite{Payne07,PayneT09}.
We show one of these is abelian, cf.\ Definition~\ref{definition.abelian GQ}.
Our main tool is the group action introduced in the appendix of~\cite{Godsil92} that Godsil attributes to Brouwer.
There as here, we actually construct the dual of a $\GQ(q,q^2)$, namely a $\GQ(q^2,q)$.

This construction involves $4$-tuples whose entries lie in the field $\bbF_{q^2}$ of $q^2$ elements, namely vectors in \smash{$\bbF_{q^2}^4$}.
Every $x\in\bbF_{q^2}$ has the property that \smash{$x^{q^2}=x$}.
Moreover, $x$ lies in the subfield $\bbF_q$ of $q$ elements if and only if $x^q=x$.
That is, $\bbF_q$ is the fixed field of the Galois group $\set{\id,\sigma}$ where $\sigma(x):=x^q$ is the Frobenius endomorphism.
If drawing an analogy between $\bbF_{q^2}$ over $\bbF_q$ and $\bbC$ over $\bbR$, $x^q$ plays the role of the complex conjugate of $x$.
In particular, the ``modulus squared" of $x\in\bbF_{q^2}$ is its \textit{field norm} $x^{q+1}=x^q x\in\bbF_q$.
Note $x^{q+1}=0$ if and only if $x=0$.
Also, $x^{q+1}=y$ has $q+1$ distinct solutions for any nonzero $y\in\bbF_q$:
letting $\alpha$ be a generator of the multiplicative group \smash{$\bbF_{q^2}^\times$} of $\bbF_{q^2}$,
we have $\bbF_q^\times=\set{\alpha^i}_{i=0}^{q-2}$;
writing $y=\alpha^{i(q+1)}$ we can take $x=\alpha^{i+j(q-1)}$ for any $j=0,\dotsc,q$.

This ``complex conjugate" analogy also suggests the following ``dot product" on $\bbF_{q^2}^4$:
\begin{equation}
\label{equation.dot product in F_q^2^4}
\bfx\cdot\bfy
=(x_1,x_2,x_3,x_4)\cdot(y_1,y_2,y_3,y_4)
:=x_1^q y_1^{}+x_2^q y_2^{}+x_3^q y_3^{}+x_4^q y_4^{}.
\end{equation}
Here, unlike the complex setting, there are many nonzero vectors that are ``orthogonal" to themselves.
Indeed, \smash{$-(x_2^{q+1}+x_3^{q+1}+x_4^{q+1})\in\bbF_q$} for any \smash{$x_2,x_3,x_4\in\bbF_{q^2}$} and so we can take $x_1$ to be any $(q+1)$th root of it.
The vertices in our $\GQ(q^2,q)$ are projective versions of lines (one-dimensional subspaces) of \smash{$\bbF_{q^2}^4$} that consist entirely of self-orthogonal vectors.
Specifically, let $[\bfx]=[x_1,x_2,x_3,x_4]$ denote the set of all nonzero scalar multiples of a given nonzero vector $\bfx=(x_1,x_2,x_3,x_4)\in\bbF_{q^2}^4$ and consider the vertex set
\begin{equation}
\label{equation.GQ(q^2,q) vertices}
\calV
:=\set{[\bfx]: \bfx\in\bbF_{q^2}^4,\ \bfx\neq\bfzero,\ \bfx\cdot\bfx=0}.
\end{equation}
This set is well-defined since the dot product~\eqref{equation.dot product in F_q^2^4} is sesquilinear.
Meanwhile, the blocks in our $\GQ(q^2,q)$ are projective versions of planes (two-dimensional subspaces) of \smash{$\bbF_{q^2}^4$} that consist entirely of self-orthogonal vectors.
To be precise, let $[\bfy,\bfz]:=\set{[\bfx]: \bfx\in\Span\set{\bfy,\bfz}}$ for any $\bfy,\bfz\in\bbF_{q^2}^4$ and note that for any self-orthogonal \smash{$\bfy,\bfz\in\bbF_{q^2}^4$},
\begin{equation}
\label{equation.GQ(q^2,q) discussion 1}
(a\bfy+b\bfz)\cdot(a\bfy+b\bfz)
=a^qb(\bfy\cdot\bfz)+b^qa(\bfz\cdot\bfy)
=a^qb(\bfy\cdot\bfz)+b^qa(\bfy\cdot\bfz)^{q}.
\end{equation}
Thus, if $\bfy\cdot\bfz=0$ then all elements in $\Span\set{\bfy,\bfz}$ are self-orthogonal.
Conversely, if \eqref{equation.GQ(q^2,q) discussion 1} is zero for all $a,b\in\bbF_{q^2}$ then $\bfy\cdot\bfz=0$ since otherwise we can take $b=1$ and rearrange~\eqref{equation.GQ(q^2,q) discussion 1} to give $a^{q-1}+(\bfy\cdot\bfz)^{q-1}=0$ for all $a\in\bbF_{q^2}$, $a\neq0$, meaning this polynomial of degree $q-1$ has $q^2-1$ roots.
This means our set of blocks is
\begin{equation}
\label{equation.GQ(q^2,q) blocks}
\calB
:=\set{[\bfy,\bfz]: \bfy,\bfz\in\bbF_{q^2}^4,\ \bfy,\bfz\neq\bfzero, [\bfy]\neq[\bfz], \ \bfy\cdot\bfy=\bfz\cdot\bfz=\bfy\cdot\bfz=0}.
\end{equation}

Note that if $\bfx=(a\bfy+b\bfz)$ where $\bfy$ and $\bfz$ are self-orthogonal then $\bfy\cdot\bfx=a(\bfy\cdot\bfy)+b(\bfy\cdot\bfz)=0$ and $\bfz\cdot\bfx=a(\bfz\cdot\bfy)+b(\bfz\cdot\bfz)=0$.
Thus, $\Span\set{\bfy,\bfz}$ is a two-dimensional subspace of
\begin{equation*}
\set{\bfy,\bfz}^\perp
:=\left\{\bfx\in\bbF_{q^2}^4:\begin{array}{r}\bfy\cdot\bfx=0\\\bfz\cdot\bfx=0\end{array}\right\}
=\operatorname{Null}\left(\left[\begin{array}{cccc}y_1^{q}&y_2^{q}&y_3^{q}&y_4^{q}\\z_1^{q}&z_2^{q}&z_3^{q}&z_4^{q}\end{array}\right]\right).
\end{equation*}
If $\bfy$ and $\bfz$ are linearly independent, the above matrix has rank two, implying $\set{\bfy,\bfz}^\perp$ is also two-dimensional and so $\Span\set{\bfy,\bfz}=\set{\bfy,\bfz}^\perp$.
In particular, we see that a vertex $[\bfx]$ is contained in a block $[\bfy,\bfz]$ if and only if $\bfy\cdot\bfx=\bfz\cdot\bfx=0$.

As mentioned above, it is known that the vertices~\eqref{equation.GQ(q^2,q) vertices} and blocks~\eqref{equation.GQ(q^2,q) blocks} form a $\GQ(q^2,q)$~\cite{PayneT09}.
For those researchers primarily interested in new ETF constructions, we have included a short, self-contained and elementary proof of this fact in Appendix~B.

To show the dual of this $\GQ(q^2,q)$ is abelian, let $\beta:=\alpha^{q-1}$.
Then $\beta$ has order $q+1$, and its powers are the $(q+1)$th roots of $1$ in $\bbF_{q^2}$.
Let the cyclic group $\bbZ_{q+1}$ act on \smash{$\bbF_{q^2}^4$} by defining
\begin{equation*}
j\bfx
=j(x_1,x_2,x_3,x_4)
:=(x_1,\beta^j x_2,\beta^j x_3,\beta^j x_4),
\quad\forall j\in\bbZ_{q+1},\, \bfx\in\bbF_{q^2}^4.
\end{equation*}
Since this action commutes with scalar-vector multiplication on \smash{$\bbF_{q^2}^4$}, we can regard it as an action on the corresponding projective space.
That is, $j[\bfx]:=[j\bfx]$ is well-defined.
Also note that since $\beta^{q+1}=1$, this action is unitary with respect to the dot product~\eqref{equation.dot product in F_q^2^4}, that is,
$(j\bfx)\cdot(j\bfy)=\bfx\cdot\bfy$.
Together, these facts imply that this action naturally applies to our GQ's vertices and blocks:
for any $j\in\bbZ_{q+1}$, $[\bfx]\in\calV$ and any $[\bfy,\bfz]\in\calB$,
\begin{align*}
j[\bfx]&:=j[x_1,x_2,x_3,x_4]=[x_1,\beta^jx_2,\beta^jx_3,\beta^jx_4]=[\beta^{-j} x_1,x_2,x_3,x_4]\in\calV,\\
j[\bfy,\bfz]&:=\set{j[\bfx]: [\bfx]\in[\bfy,\bfz]}=\set{[\bfx]: \bfx\in[j\bfy,j\bfz]}=[j\bfy,j\bfz]\in\calB.
\end{align*}

We denote the orbit of any $[\bfx]\in\calV$ under this action as \smash{$\orb[\bfx]:=\set{j[\bfx]}_{j\in\bbZ_{q+1}}$}.
Note that if $\bfx=(x_1,x_2,x_3,x_4)$ where $x_1=0$, this orbit is just the single vertex $[\bfx]$.
If instead $x_1\neq0$, the fact that $\bfx$ is self-orthogonal implies $(x_2,x_3,x_4)\neq(0,0,0)$ and so the vertices $\set{j[\bfx]}_{j\in\bbZ_{q+1}}$ are all distinct.
In particular, the only fixed points of the group action are the vertices
\begin{equation}
\label{equation.definition of GQ(q^2,q) ovoid}
\calO:=\set{[\bfx]\in\calV: x_1=0}.
\end{equation}

Now recall that in order to satisfy Definition~\ref{definition.abelian GQ},
our $\GQ(q,q^2)$ necessarily contains a spread.
This means its dual $(\calV,\calB)$ necessarily contains an \textit{ovoid},
namely a set of vertices $\calO$ with the property that any block contains exactly one of them.
As we now explain, the set~\eqref{equation.definition of GQ(q^2,q) ovoid} is an ovoid.
Indeed, every block $[\bfy,\bfz]$ contains a member of $\calO$, namely either $[\bfy]$ or $[a\bfy+\bfz]$ for some choice of $a\in\bbF_{q^2}$.
Moreover, this member is unique: no self-orthogonal vector can have support of size one,
meaning if $[\bfy],[\bfz]\in\calO$ are distinct but lie in a common block,
performing Gaussian elimination on them produces a basis for $\Span\set{\bfy,\bfz}$ of the form $\set{(0,1,0,a),(0,0,1,b)}$ where $a^{q+1}=b^{q+1}=-1$;
however, since they are a basis for $\Span\set{\bfy,\bfz}$, they are also necessarily orthogonal, a contradiction.

Since there are $(t+1)(st+1)=(q+1)(q^3+1)$ blocks total, each containing a unique vertex in $\calO$,
and since each vertex is contained in $t+1=q+1$ blocks, the number of vertices in $\calO$ is $q^3+1$.
The vertices in $\calV\cap\calO^\rmc$ are of the form $[1,x_1,x_2,x_3]$ where $x_1^{q+1}+x_2^{q+1}+x_3^{q+1}=-1$.
Since the total number of vertices in $\calV$ is $(s+1)(st+1)=(q^2+1)(q^3+1)$, there are exactly $q^2(q^3+1)$ vertices in $\calV\cap\calO^\rmc$.
Moreover, since the orbits of a group action form a partition of the set on which it acts,
and since any orbit generated from a nonovoid vertex has cardinality $q+1$,
the nonovoid vertices $\calV\cap\calO^\rmc$ are partitioned into $q^2(q^2-q+1)$ orbits of size $q+1$.
From each such orbit, pick a representative vertex $[\bfx]$.
This arbitrary choice establishes an ordering amongst the vertices in this orbit:
we regard $[1,\beta^ix_1,\beta^ix_2,\beta^ix_3]$ as the $i$th vertex in $\orb[\bfx]$.

We order the blocks as well.
In particular, for any of the $q^3+1$ ovoid vertices $\bfy$,
fix any $\bfz=(1,z_2,z_3,z_4)$ that is orthogonal to $\bfy$.
Since $\bfy$ is a fixed point of the group action, $[\bfy,j\bf z]$ is a valid block for any $j\in\bbZ_{q+1}$.
Moreover, these blocks are all distinct since
\begin{equation*}
0
=(j\bfz)\cdot(j'\bfz)
=1+\beta^{j'-j}(z_1^{q+1}+z_2^{q+1}+z_3^{q+1})
=1-\beta^{j'-j}
\end{equation*}
if and only if $j=j'$.
This means every block that contains $\bfy$ is of this form,
and we regard $[\bfy, j\bfz]$ as the $j$th block that contains $\bfy$.

Now let $\bfY$ be a $q^2(q^2-q+1)\times(q^3+1)$ array of $(q+1)\times(q+1)$ matrices,
where the rows and columns of this array are indexed by our nonovoid orbit representatives $[\bfx]$ and ovoid vertices $[\bfy]$, respectively.
For any such $[\bfx]$ and $[\bfy]$ and any $i,j\in\bbZ_{q+1}$,
let the $(i,j)$th entry of $(q+1)\times(q+1)$ matrix $\bfY([\bfx],[\bfy])$ be $1$ if the $i$th vertex in $\orb[\bfx]$ lies in the $j$th block that contains $[\bfy]$, that is, if $i\bfx\in\Span\set{\bfy,j\bfz}$, and otherwise be $0$.
This happens if and only if both $\bfx\cdot\bfy=0$ and
\begin{equation*}
0
=(i\bfx)\cdot(j\bfz)
=1+\beta^{j-i}(x_2^qz_2^{}+x_3^qz_3^{}+x_4^qz_4^{}),
\end{equation*}
which requires $j-i$ to be some fixed constant in $\bbZ_{q+1}$.
As such, $\bfY([\bfx],[\bfy])$ is either zero or a permutation matrix which is circulant over $\bbZ_{q+1}$.
At the same time, note the matrix
\begin{equation*}
\bfZ
=\left[\begin{array}{c}\bfI_{q^3+1}^{}\otimes\bfone_{q+1}^\rmT\\\bfY\end{array}\right]
\end{equation*}
is an incidence matrix for the dual of our $\GQ(q^2,q)$;
each of the first $q^3+1$ rows of $\bfZ$ corresponds to a vertex $[\bfy]$ in $\calO$,
and the top of $\bfZ$ has this form since $[\bfy]\in[\bfy,j\bfz]$ for all $j\in\bbZ_{q+1}$.
Thus, by Theorem~\ref{theorem.abelian GQ}, the dual of our $\GQ(q^2,q)$ is indeed an abelian $\GQ(q,q^2)$.
Putting this together with other facts from Theorems~\ref{theorem.polyphase BIBD ETF} and~\ref{theorem.abelian GQ} then immediately gives the following result:

\begin{theorem}
\label{theorem.(q^3+1,q+1,q+1) polyphase BIBD ETFs}
For any prime power $q$, there exists an abelian $\GQ(q,q^2)$ over $\bbZ_{q+1}$.
As such, letting $\bfPhi(z)$ be the corresponding $(q^3+1,q+1,q+1)$-polyphase BIBD ETF:
\begin{enumerate}
\alphi
\item
For any nontrivial character $\gamma$,
$\bfPhi(\gamma)$ is a phased BIBD ETF with parameters~\eqref{equation.new ETF parameters}.
\item
When $q$ is odd, $\gamma$ can be chosen so that $\bfPhi(\gamma)$ is a real ETF.
\item
$\bfPhi(z)^*\bfPhi(z)-(q+1)\bfI$ is an abelian $(q^3+1,q+1,q^2-1)$-DRACKN.
\end{enumerate}
\end{theorem}

For $q$ odd, the existence of real ETFs with parameters~\eqref{equation.new ETF parameters} was already known.
In fact, the existence of the corresponding SRGs was proven in~\cite{Godsil92},
and it was this realization that started our investigation into this line of research a few years ago.
For $q$ even, the situation is more interesting.
In particular, complex ETFs with parameters~\eqref{equation.new ETF parameters} do arise as Naimark complements of Steiner ETFs arising from finite projective planes of order $q-1$~\cite{FickusMT12}.
However, finite projective planes are only known to exist when their order is a power of a prime.
This means the complex ETFs given by Theorem~\ref{theorem.(q^3+1,q+1,q+1) polyphase BIBD ETFs} are new whenever $q=2^j$ but $q-1$ is not an odd prime power, such as when $q=16$.
This happens infinitely often.
For example, since $2^j-1\equiv(-1)^j-1\equiv0\bmod 3$ whenever $j$ is even
but $2^j-1\equiv 0\bmod 9$ only when $j\equiv 0\bmod 6$, this means $2^j-1$ is divisible by $3$ but not by $9$ whenever $j\equiv 2,4\bmod 6$, implying it is not a prime power for any such $j$ greater than $2$.

Theorem~\ref{theorem.(q^3+1,q+1,q+1) polyphase BIBD ETFs} accounts for half of the constructions of phased BIBD ETFs given in Table~\ref{table.phased BIBD ETFs};
the other half arise from Theorem~\ref{theorem.(q^2,q,q) polyphase BIBD ETFs}.
Comparing against the constructions given in~\cite{CoutinkhoGSZ16}, we also note that the abelian DRACKNs given by Theorem~\ref{theorem.(q^3+1,q+1,q+1) polyphase BIBD ETFs} also seem to be new for all prime powers $q>2$, having parameter $\delta=n-fc-2=-q(q-1)$.

\begin{figure}
\begin{tiny}
\begin{equation*}
\setlength{\arraycolsep}{2pt}
\settowidth{\NewLengthOne}{$z^2$}
\left[\begin{array}{cccccccccccccccccccccccccccc}
  1&  1&  1&  1&  0&  0&  0&  0&  0&  0&  0&  0&  0&  0&  0&  0&  0&  0&  0&  0&  0&  0&  0&  0&\makebox[\NewLengthOne][c]{0}&\makebox[\NewLengthOne][c]{0}&\makebox[\NewLengthOne][c]{0}&\makebox[\NewLengthOne][c]{0}\\
  0&  0&  0&  0&  1&  1&  1&  1&  0&  0&  0&  0&  0&  0&  0&  0&  0&  0&  0&  0&  0&  0&  0&  0&  0&  0&  0&  0\\
  0&  0&  0&  0&  0&  0&  0&  0&  1&  1&  1&  1&  0&  0&  0&  0&  0&  0&  0&  0&  0&  0&  0&  0&  0&  0&  0&  0\\
  0&  0&  0&  0&  0&  0&  0&  0&  0&  0&  0&  0&  1&  1&  1&  1&  0&  0&  0&  0&  0&  0&  0&  0&  0&  0&  0&  0\\
  0&  0&  0&  0&  0&  0&  0&  0&  0&  0&  0&  0&  0&  0&  0&  0&  1&  1&  1&  1&  0&  0&  0&  0&  0&  0&  0&  0\\
  0&  0&  0&  0&  0&  0&  0&  0&  0&  0&  0&  0&  0&  0&  0&  0&  0&  0&  0&  0&  1&  1&  1&  1&  0&  0&  0&  0\\
  0&  0&  0&  0&  0&  0&  0&  0&  0&  0&  0&  0&  0&  0&  0&  0&  0&  0&  0&  0&  0&  0&  0&  0&  1&  1&  1&  1\\
  0&  0&  0&  0&  0&  0&  0&  0&z^2&  0&  0&  0&  0&z^2&  0&  0&  0&  0&z^2&  0&  0&  0&  0&z^2&  0&  0&  0&  0\\
  0&  0&  0&  0&  0&  0&  0&  0&  0&  z&  0&  0&  0&  0&  z&  0&  0&  0&  0&  z&  z&  0&  0&  0&  0&  0&  0&  0\\
  0&  0&  0&  0&  0&  0&  0&  0&  0&  0&  1&  0&  0&  0&  0&  1&  1&  0&  0&  0&  0&  1&  0&  0&  0&  0&  0&  0\\
  0&  0&  0&  0&  0&  0&  0&  0&  0&  0&  0&z^3&z^3&  0&  0&  0&  0&z^3&  0&  0&  0&  0&z^3&  0&  0&  0&  0&  0\\
  0&  0&  0&  0&  0&  0&  0&  0&  0&  0&z^2&  0&  0&  0&z^3&  0&  0&  0&  1&  0&  0&  0&  z&  0&  0&  0&  0&  0\\
  0&  0&  0&  0&  0&  0&  0&  0&  0&  0&  0&  z&  0&  0&  0&z^2&  0&  0&  0&z^3&  0&  0&  0&  1&  0&  0&  0&  0\\
  0&  0&  0&  0&  0&  0&  0&  0&  1&  0&  0&  0&  z&  0&  0&  0&z^2&  0&  0&  0&z^3&  0&  0&  0&  0&  0&  0&  0\\
  0&  0&  0&  0&  0&  0&  0&  0&  0&z^3&  0&  0&  0&  1&  0&  0&  0&  z&  0&  0&  0&z^2&  0&  0&  0&  0&  0&  0\\
  1&  0&  0&  0&  0&  0&  0&  1&  0&  0&  0&  0&  0&  0&  0&  z&  0&  0&  0&  0&  1&  0&  0&  0&  0&  0&  0&  0\\
z^3&  0&  0&  0&  1&  0&  0&  0&  0&  1&  0&  0&  0&  0&  0&  0&  z&  0&  0&  0&  0&  0&  0&  0&  0&  0&  0&  0\\
z^2&  0&  0&  0&  0&  1&  0&  0&  0&  0&  0&  0&  0&  0&  1&  0&  0&  0&  0&  0&  0&  z&  0&  0&  0&  0&  0&  0\\
  z&  0&  0&  0&  0&  0&  1&  0&  0&  0&  z&  0&  0&  0&  0&  0&  0&  0&  0&  1&  0&  0&  0&  0&  0&  0&  0&  0\\
  0&  1&  0&  0&  0&  0&  0&z^3&  0&  0&  0&  1&  0&  0&  0&  0&z^3&  0&  0&  0&  0&  0&  0&  0&  0&  0&  0&  0\\
  0&z^3&  0&  0&z^3&  0&  0&  0&  0&  0&  0&  0&  1&  0&  0&  0&  0&  0&  0&  0&  0&z^3&  0&  0&  0&  0&  0&  0\\
  0&z^2&  0&  0&  0&z^3&  0&  0&  0&  0&z^3&  0&  0&  0&  0&  0&  0&  1&  0&  0&  0&  0&  0&  0&  0&  0&  0&  0\\
  0&  z&  0&  0&  0&  0&z^3&  0&  0&  0&  0&  0&  0&  0&  0&z^3&  0&  0&  0&  0&  0&  0&  1&  0&  0&  0&  0&  0\\
  0&  0&  1&  0&  0&  0&  0&z^2&  0&  0&  0&  0&z^2&  0&  0&  0&  0&  0&  0&  0&  0&  0&  0&z^3&  0&  0&  0&  0\\
  0&  0&z^3&  0&z^2&  0&  0&  0&z^3&  0&  0&  0&  0&  0&  0&  0&  0&z^2&  0&  0&  0&  0&  0&  0&  0&  0&  0&  0\\
  0&  0&z^2&  0&  0&z^2&  0&  0&  0&  0&  0&  0&  0&z^3&  0&  0&  0&  0&  0&  0&  0&  0&z^2&  0&  0&  0&  0&  0\\
  0&  0&  z&  0&  0&  0&z^2&  0&  0&  0&  0&z^2&  0&  0&  0&  0&  0&  0&z^3&  0&  0&  0&  0&  0&  0&  0&  0&  0\\
  0&  0&  0&  1&  0&  0&  0&  z&  z&  0&  0&  0&  0&  0&  0&  0&  0&  0&  0&z^2&  0&  0&  0&  0&  0&  0&  0&  0\\
  0&  0&  0&z^3&  z&  0&  0&  0&  0&  0&  0&  0&  0&  z&  0&  0&  0&  0&  0&  0&z^2&  0&  0&  0&  0&  0&  0&  0\\
  0&  0&  0&z^2&  0&  z&  0&  0&  0&z^2&  0&  0&  0&  0&  0&  0&  0&  0&  z&  0&  0&  0&  0&  0&  0&  0&  0&  0\\
  0&  0&  0&  z&  0&  0&  z&  0&  0&  0&  0&  0&  0&  0&z^2&  0&  0&  0&  0&  0&  0&  0&  0&  z&  0&  0&  0&  0\\
  0&  0&  0&  1&  0&  0&  0&  0&  0&  0&  0&  0&  0&  0&  0&z^3&  0&  z&  0&  0&  0&  0&  0&  0&  1&  0&  0&  0\\
  0&  0&  0&z^3&  0&  0&  0&  0&  0&  0&  0&  0&  0&  0&  0&  0&z^3&  0&  0&  0&  0&  0&  z&  0&  0&  1&  0&  0\\
  0&  0&  0&z^2&  0&  0&  0&  0&  0&  0&  0&  z&  0&  0&  0&  0&  0&  0&  0&  0&  0&z^3&  0&  0&  0&  0&  1&  0\\
  0&  0&  0&  z&  0&  0&  0&  0&  0&  0&z^3&  0&  z&  0&  0&  0&  0&  0&  0&  0&  0&  0&  0&  0&  0&  0&  0&  1\\
  0&  0&  0&  0&  0&  0&  1&  0&  0&  0&  0&  0&  0&  0&  0&  0&  0&z^3&  0&  0&z^2&  0&  0&  0&  0&  0&  0&  1\\
  0&  0&  0&  0&  0&  0&  0&  1&  0&z^2&  0&  0&  0&  0&  0&  0&  0&  0&  0&  0&  0&  0&z^3&  0&  1&  0&  0&  0\\
  0&  0&  0&  0&  1&  0&  0&  0&  0&  0&  0&z^3&  0&  0&z^2&  0&  0&  0&  0&  0&  0&  0&  0&  0&  0&  1&  0&  0\\
  0&  0&  0&  0&  0&  1&  0&  0&  0&  0&  0&  0&z^3&  0&  0&  0&  0&  0&  0&z^2&  0&  0&  0&  0&  0&  0&  1&  0\\
  1&  0&  0&  0&  0&  0&  0&  0&  0&  0&  0&z^2&  0&  1&  0&  0&  0&  0&  0&  0&  0&  0&  0&  0&  0&  0&  0&  1\\
z^3&  0&  0&  0&  0&  0&  0&  0&  0&  0&  0&  0&z^2&  0&  0&  0&  0&  0&  1&  0&  0&  0&  0&  0&  1&  0&  0&  0\\
z^2&  0&  0&  0&  0&  0&  0&  0&  0&  0&  0&  0&  0&  0&  0&  0&  0&z^2&  0&  0&  0&  0&  0&  1&  0&  1&  0&  0\\
  z&  0&  0&  0&  0&  0&  0&  0&  1&  0&  0&  0&  0&  0&  0&  0&  0&  0&  0&  0&  0&  0&z^2&  0&  0&  0&  1&  0\\
  0&  0&  0&  0&  0&  0&z^3&  0&  0&  0&  0&  0&  0&z^2&  0&  0&  z&  0&  0&  0&  0&  0&  0&  0&  0&  0&  1&  0\\
  0&  0&  0&  0&  0&  0&  0&z^3&  0&  0&  0&  0&  0&  0&  0&  0&  0&  0&z^2&  0&  0&  z&  0&  0&  0&  0&  0&  1\\
  0&  0&  0&  0&z^3&  0&  0&  0&  0&  0&  z&  0&  0&  0&  0&  0&  0&  0&  0&  0&  0&  0&  0&z^2&  1&  0&  0&  0\\
  0&  0&  0&  0&  0&z^3&  0&  0&z^2&  0&  0&  0&  0&  0&  0&  z&  0&  0&  0&  0&  0&  0&  0&  0&  0&  1&  0&  0\\
  0&  1&  0&  0&  0&  0&  0&  0&  0&z^3&  0&  0&  0&  0&  0&  0&  0&  0&  0&  0&  0&  0&  0&  z&  0&  0&  1&  0\\
  0&z^3&  0&  0&  0&  0&  0&  0&  z&  0&  0&  0&  0&  0&z^3&  0&  0&  0&  0&  0&  0&  0&  0&  0&  0&  0&  0&  1\\
  0&z^2&  0&  0&  0&  0&  0&  0&  0&  0&  0&  0&  0&  z&  0&  0&  0&  0&  0&z^3&  0&  0&  0&  0&  1&  0&  0&  0\\
  0&  z&  0&  0&  0&  0&  0&  0&  0&  0&  0&  0&  0&  0&  0&  0&  0&  0&  z&  0&z^3&  0&  0&  0&  0&  1&  0&  0\\
  0&  0&  0&  0&  0&  0&z^2&  0&  0&  z&  0&  0&  1&  0&  0&  0&  0&  0&  0&  0&  0&  0&  0&  0&  0&  1&  0&  0\\
  0&  0&  0&  0&  0&  0&  0&z^2&  0&  0&  0&  0&  0&  0&  z&  0&  0&  1&  0&  0&  0&  0&  0&  0&  0&  0&  1&  0\\
  0&  0&  0&  0&z^2&  0&  0&  0&  0&  0&  0&  0&  0&  0&  0&  0&  0&  0&  0&  z&  0&  0&  1&  0&  0&  0&  0&  1\\
  0&  0&  0&  0&  0&z^2&  0&  0&  0&  0&  0&  1&  0&  0&  0&  0&  0&  0&  0&  0&  z&  0&  0&  0&  1&  0&  0&  0\\
  0&  0&  1&  0&  0&  0&  0&  0&  0&  0&  0&  0&  0&  0&  0&  0&  0&  0&  0&  1&  0&z^2&  0&  0&  0&  1&  0&  0\\
  0&  0&z^3&  0&  0&  0&  0&  0&  0&  0&z^2&  0&  0&  0&  0&  0&  0&  0&  0&  0&  1&  0&  0&  0&  0&  0&  1&  0\\
  0&  0&z^2&  0&  0&  0&  0&  0&  0&  1&  0&  0&  0&  0&  0&z^2&  0&  0&  0&  0&  0&  0&  0&  0&  0&  0&  0&  1\\
  0&  0&  z&  0&  0&  0&  0&  0&  0&  0&  0&  0&  0&  0&  1&  0&z^2&  0&  0&  0&  0&  0&  0&  0&  1&  0&  0&  0\\
  0&  0&  0&  0&  0&  0&  z&  0&z^3&  0&  0&  0&  0&  0&  0&  0&  0&  0&  0&  0&  0&  1&  0&  0&  1&  0&  0&  0\\
  0&  0&  0&  0&  0&  0&  0&  z&  0&  0&  1&  0&  0&z^3&  0&  0&  0&  0&  0&  0&  0&  0&  0&  0&  0&  1&  0&  0\\
  0&  0&  0&  0&  z&  0&  0&  0&  0&  0&  0&  0&  0&  0&  0&  1&  0&  0&z^3&  0&  0&  0&  0&  0&  0&  0&  1&  0\\
  0&  0&  0&  0&  0&  z&  0&  0&  0&  0&  0&  0&  0&  0&  0&  0&  1&  0&  0&  0&  0&  0&  0&z^3&  0&  0&  0&  1
\end{array}\right]
\end{equation*}
\end{tiny}
\caption{
\label{figure.(28,4,4)-polyphase BIBD ETF}
A $(28,4,4)$-polyphase BIBD ETF $\bfPhi(z)$ produced according to Theorem~\ref{theorem.(q^3+1,q+1,q+1) polyphase BIBD ETFs}.
By Theorem~\ref{theorem.polyphase BIBD ETF}, letting $z=1$ gives the incidence matrix of a $\BIBD(28,4,1)$,
while letting $z$ be either $\rmi$, $-1$ or $-\rmi$ produces an ETF $\bfPhi(\gamma)$ with $(d,n)=(21,28)$.
When $z=-1$, this ETF is real.
By Theorem~\ref{theorem.abelian GQ}, the $252\times 112$ filter bank matrix $\bfPhi(\bfT)$ obtained by identifying each entry of the $63\times 28$ matrix $\bfPhi(z)$ with a $4\times 4$ circulant matrix via~\eqref{equation.definition of x(T)} is the incidence matrix of the incidence structure obtained from removing a spread from an abelian $\GQ(3,9)$.
More generally, Theorem~\ref{theorem.(q^3+1,q+1,q+1) polyphase BIBD ETFs} gives $(q^3+1,q+1,q+1)$-polyphase BIBD ETFs for any prime power $q$.
These ETFs are demonstrably new whenever $q$ is an even prime power with the property that $q-1$ is not an odd prime power, which happens infinitely often.
}
\end{figure}

An example of a $(28,4,4)$-polyphase BIBD ETF arising from Theorem~\ref{theorem.(q^3+1,q+1,q+1) polyphase BIBD ETFs} when $q=3$ is given in Figure~\ref{figure.(28,4,4)-polyphase BIBD ETF}.
As discussed earlier, from~\cite{CoutinkhoGSZ16} we know that polyphase BIBD ETFs with these parameters are the only ones capable of generating real ETFs whose Naimark complements achieve the real Gerzon bound.
We also point out that in this example the underlying $\BIBD(28,4,1)$ is nicely arranged, with its top seven rows corresponding to a \textit{parallel class}, namely a set of blocks in the BIBD that partition its vertex set.
This was accomplished by arranging our $q^3+1$ ovoid vertices into $q^2-q+1$ subsets of size $q+1$ according to which member of a special set of $q^2-q+1$ nonovoid orbit representatives they form a block with.

To be precise,
one can show that all blocks in~\eqref{equation.GQ(q^2,q) blocks} are either of the form
\begin{equation*}
[a,b;j]
:=\bigset{[\bfx]: \bfx\in\Span\set{(1,0,a,b),(0,1,-\beta^j b^q,\beta^j a^q)}},\quad j=0,\dotsc,q,
\end{equation*}
where $a^{q+1}+b^{q+1}=-1$ or of the form
\begin{equation*}
[a;j]
:=\bigset{[\bfx]: \bfx\in\Span\set{(1,a,0,0),(0,0,1,\beta^j a)}},\quad j=0,\dotsc,q,
\end{equation*}
where $a^{q+1}=-1$.
Under this organization of our blocks, the $\BIBD(q^3+1,q+1,1)$ will always have an immediately identifiable parallel class provided we choose as many orbit representatives as possible of the form $[1,0,a,b]$ and $[1,a,0,0]$.

This was not done for purely cosmetic reasons.
One benefit to writing our blocks in this way is that it naturally parameterizes the $q^2+1$ vertices that any given block contains,
thereby speeding up our MATLAB implementation of this ETF construction technique.
For example, a vertex $[c,d,e,f]$ lies in $[a,b;j]$ if and only if
\begin{equation*}
(c,d,e,f)
=c(1,0,a,b)+d(0,1,-\beta^j b^q,\beta^j a^q),\text{ i.e. }
\left[\begin{array}{c}e\\f\end{array}\right]
=\left[\begin{array}{rr}a&-\beta^j b^q\\b&\beta^j a^q\end{array}\right]\left[\begin{array}{r}c\\d\end{array}\right].
\end{equation*}
The above matrix is invertible since its determinant is $\beta^j(a^{q+1}+b^{q+1})=-\beta^j\neq0$.
Since $(c,d,e,f)$ is nonzero and only unique up to scalar multiples, this implies $(c,d)$ is necessarily nonzero and only unique up to scalar multiples.
Taking $(c,d)=(0,1)$ gives the ovoid vertex while taking $(c,d)=(1,d)$ where $d\in\bbF_{q^2}$ is arbitrary gives the $q^2$ remaining vertices in this block.
Similarly, $[c,d,e,f]$ lies in $[a;j]$ if and only if
\begin{equation*}
(c,d,e,f)
=c(1,a,0,0)+e(0,0,1,\beta^j a),\text{ i.e. }
\begin{array}{l}d=ac,\\ f=\beta^ja e,\end{array}
\end{equation*}
and taking $(c,e)=(0,1)$ gives the ovoid vertex while taking $(c,e)=(1,e)$ where $e\in\bbF_{q^2}$ gives the $q^2$ nonovoid vertices.

Writing our blocks in this way also benefits our theory:
whenever the columns of $\bfPhi$ form an ETF for their span and $\abs{\bfPhi}^2$ is the incidence matrix of a $\BIBD(v,k,1)$ that contains a parallel class,
we can multiply these columns by unimodular scalars so as to assume, without loss of generality, that the all-ones vector lies in the row space of $\bfPhi$.
As discussed in~\cite{FickusMJPW15}, this means these ETFs have \textit{axial symmetry},
meaning $\bfone$ is an eigenvector for the Gram matrix with a nonzero eigenvalue.
If we further know that $\bfPhi$ is real,
applying Theorem~4.2 of~\cite{FickusMJPW15} to it produces an SRG with parameters
\begin{equation*}
v_{\SRG}=v,
\quad
k_{\SRG}=\tfrac12(v+k-2),
\quad
\lambda_{\SRG}=\tfrac14(v+3k-r-7),
\quad
\mu_{\SRG}=\tfrac14(v+k+r-1).
\end{equation*}
In particular, whenever $q$ is an odd prime power, applying this fact to the real ETFs produced by Theorem~\ref{theorem.(q^3+1,q+1,q+1) polyphase BIBD ETFs} gives SRGs with parameters
$(q^3+1,\frac12q(q^2+1),\frac14(q-1)(q^2+3),\frac14(q+1)(q^2+1))$.
By consulting~\cite{Brouwer07,Brouwer15}, we see that SRGs with these parameters were already known~\cite{Taylor74}.

\appendix

\section{Proof of Theorem~\ref{theorem.phased BIBD necessary conditions}}

We first show that $u$ is a nonnegative integer.
By~\eqref{equation.phased BIBD ETF dimension},
the dimension of the span of any Naimark complement of the columns of $\bfPhi$ is
\begin{equation*}
v-d
=v-\frac{vr}{r+k-1}
=\frac{v(k-1)}{r+k-1}.
\end{equation*}
By~\eqref{equation.BIBD properties} and \eqref{equation.phased BIBD ETF parameter} we thus have
\begin{equation}
\label{equation.proof of phased BIBD necessary conditions 0}
v-d
=\frac{v(k-1)^2}{v+k(k-2)}
=(k-1)^2-\frac{k(k-1)^2(k-2)}{v+k(k-2)}
=(k-1)^2-u.
\end{equation}
Thus, $u$ is an integer.
Moreover, since $k\geq2$ then $u\geq0$.

We next show that $u$ divides $k(k-1)(k-2)$ and $r(k-1)(k-2)$.
When $k=2$, $u=0$ and this is immediate.
For $k\geq 3$, $u>0$ and solving for $v$ in~\eqref{equation.phased BIBD ETF parameter} gives
\begin{equation}
\label{equation.proof of phased BIBD necessary conditions 1}
v=\frac1uk(k-1)^2(k-2)-k(k-2).
\end{equation}
In particular, $v-1=\tfrac1uk(k-1)^2(k-2)-(k-1)^2$ and so
\begin{equation}
\label{equation.proof of phased BIBD necessary conditions 2}
r=\frac{v-1}{k-1}=\frac1uk(k-1)(k-2)-(k-1).
\end{equation}
Since $k$ and $r$ are integers, this implies $u$ divides $k(k-1)(k-2)$.
Moreover, combining~\eqref{equation.proof of phased BIBD necessary conditions 1} and~\eqref{equation.proof of phased BIBD necessary conditions 2} gives the following expression for $\frac vk$:
\begin{equation*}
\frac{v}{k}
=\frac1u(k-1)^2(k-2)-(k-2)
=r+1-\frac1u(k-1)(k-2),
\end{equation*}
which in turn implies
\begin{equation*}
b
=\frac{v}{k}r
=\biggl[r+1-\frac1u(k-1)(k-2)\biggr]r
=r(r+1)-\frac1ur(k-1)(k-2).
\end{equation*}
Since $b$ and $r$ are integers, this implies $u$ also divides $r(k-1)(k-2)$.

We next show $u\leq\frac12(k-1)(k-2)$.
Fisher's inequality states $v\leq b$.
Since $bk=vr$, this can be restated as $k\leq r$.
If $r=k$ then $v=r(k-1)+1=k^2-k+1$ and so~\eqref{equation.phased BIBD ETF dimension} becomes
\begin{equation*}
d
=\frac{vr}{r+k-1}
=\frac{k(k^2-k+1)}{2k-1}
=\frac18(4k^2-2k+3)+\frac3{8(2k-1)},
\end{equation*}
which is only an integer when $k=2$, at which point $u=0=\frac12(k-1)(k-2)$.
As such, it suffices to consider the case where $r\geq k+1$.
Here, \eqref{equation.proof of phased BIBD necessary conditions 2} implies
\begin{equation*}
k+1
\leq r
=\frac1uk(k-1)(k-2)-(k-1),
\end{equation*}
and so $u\leq\frac12(k-1)(k-2)$, as claimed.

Next note~\eqref{equation.phased BIBD ETF dimension} implies $\frac vd-1=\frac{k-1}r$, and multiplying and dividing $v-1=r(k-1)$ by it gives~\eqref{equation.phased BIBD ETF reciprocal Welch bound}.
To prove the stated necessary conditions on real $\bfPhi$, it helps to first prove (a).

For (a), note that since $k\geq2$, $u$ as defined in~\eqref{equation.phased BIBD ETF parameter} is zero if and only if $k=2$.
Moreover, in this case~\eqref{equation.proof of phased BIBD necessary conditions 0} becomes $v-d=1$, that is, $v=d+1$.
Conversely, if $v\leq d+1$ then~\eqref{equation.proof of phased BIBD necessary conditions 0} gives $1\geq(k-1)^2-u$,
that is, $u\geq k(k-2)$, which violates the necessary condition $u\leq\frac12(k-1)(k-2)$ unless $k=2$.
That is, we have $v\geq d+1$ in general, and moreover $v=d+1$ if and only if $k=2$.

Further note that in general, \eqref{equation.proof of phased BIBD necessary conditions 0} and Fisher's inequality ($r\geq k$) imply that $d\geq 2$:
\begin{equation*}
d
=v-(k-1)^2+u
=(r-k+1)(k-1)+1+u
\geq k+u
\geq 2.
\end{equation*}
(This also follows immediately from the fact that the columns of a phased BIBD are not collinear.)
When taken together with Theorem~\ref{theorem.phased BIBD ETFs}, we see that the $v$ and $d$ parameters of a phased BIBD ETF satisfy $1\leq d\leq v-1$ with $v<2d$, where $d=v-1$ if and only if $k=2$.
In particular, if $k>2$ and $\bfPhi$ is real, then its columns satisfy the hypotheses of Theorem~A of~\cite{SustikTDH07}, implying the numbers in~\eqref{equation.phased BIBD ETF reciprocal Welch bound} are odd.
That is, $k$ is even and $r$ is odd, which in turn implies $v=r(k-1)+1$ is even.

What remains to be shown are the equivalences stated in (b), (c) and (d).
To be clear, in each we are assuming that a phased BIBD ETF arising from a $\BIBD(v,k,1)$ exists,
and are simply identifying relationships between its parameters.
We omit the proof of (b) since it is similar to, but less difficult than, the proof of (c).
For (c), note~\eqref{equation.phased BIBD ETF parameter} immediately implies that $u=2$ if and only if $v=\binom{k(k-2)}{2}$.
Moreover, in this case, \eqref{equation.proof of phased BIBD necessary conditions 0} gives
\begin{equation*}
\binom{v-d+1}{2}
=\binom{(k-1)^2-u+1}{2}
=\binom{k(k-2)}{2}
=v.
\end{equation*}
Conversely, if $v=\binom{v-d+1}{2}$ then letting $s=k-1$, \eqref{equation.proof of phased BIBD necessary conditions 0}
and \eqref{equation.phased BIBD ETF parameter} give
\begin{equation*}
v
=\binom{v-d+1}{2}
=\binom{s^2-u+1}{2}
=\frac12\biggl[s^2+1-\frac{s^2(s^2-1)}{v+(s^2-1)}\biggr]\biggl[s^2-\frac{s^2(s^2-1)}{v+(s^2-1)} \biggr].
\end{equation*}
Multiplying by $[v+(s^2-1)]^2$ and simplifying gives $0=(v-1)[2v-(s^2-1)(s^2-2)]$.
Since we assume $v>k\geq 2$, this implies $v=\binom{s^2-1}{2}=\binom{k(k-2)}{2}$.

For (d), if $v=k^2$ then~\eqref{equation.phased BIBD ETF parameter} gives $u=\frac12(k-1)(k-2)$.
Conversely, if $u=\frac12(k-1)(k-2)$ where $k>2$ then~\eqref{equation.proof of phased BIBD necessary conditions 1} gives $v=k^2$.
(If $k=2$, \eqref{equation.proof of phased BIBD necessary conditions 1} is invalid, and indeed the conclusion is false since for any $v\geq 3$ there exists a phased $\BIBD(v,2,1)$ ETF, cf.\ Example~\ref{example.regular simplices as phased BIBD ETFs}).
Next, if $k>2$ then taking $v=k^2$ we have $v>4$ and \eqref{equation.proof of phased BIBD necessary conditions 0} gives
\begin{equation*}
v-d
=(k-1)^2-u
=(k-1)^2-\tfrac12(k-1)(k-2)
=\tfrac12(k^2-k)
=\tfrac12(v-v^{\frac12}).
\end{equation*}
Conversely, if $v-d=\frac12(v-v^{\frac12})$ with $v>4$ then~\eqref{equation.proof of phased BIBD necessary conditions 0} and~\eqref{equation.phased BIBD ETF parameter} gives
\begin{equation*}
\tfrac12(v-v^{\frac12})
=v-d
=(k-1)^2-u
=(k-1)^2-\frac{k(k-1)^2(k-2)}{v+k(k-2)}.
\end{equation*}
Multiplying by $v+k(k-2)$ and simplifying then gives
\begin{equation*}
0
=v^{\frac32}-v-(k^2-2k+2)v^{\frac12}-k(k-2)
=(v^{\frac12}-k)(v^{\frac12}+1)[v^{\frac12}+(k-2)]
\end{equation*}
and so $v=k^2$.
Since $v>4$, $k>2$.

\section{A proof that~\eqref{equation.GQ(q^2,q) vertices} and~\eqref{equation.GQ(q^2,q) blocks} define a $\GQ(q^2,q)$.}

Here, we give an elementary proof that~\eqref{equation.GQ(q^2,q) vertices} and~\eqref{equation.GQ(q^2,q) blocks} form a $\GQ(q^2,q)$.
We do so by verifying they satisfy properties (i)--(v) of a GQ given at the beginning of Subsection~\ref{subsection.GQ}.
Note any two distinct blocks have at most one vertex in common, since the intersection of two distinct planes in \smash{$\bbF_{q^2}^4$} is either $\set{\bfzero}$ or a line.
Any two distinct vertices $[\bfy],[\bfz]$ are contained in at most one block:
the lines spanned by $\bfy$ and $\bfz$ determine a unique plane, and the projective version of this plane is a member of $\calB$ if and only if $\bfy\cdot\bfz=0$.
Also, block $[\bfy,\bfz]$ contains exactly $q^2+1$ vertices, namely $[\bfz]$ and $[\bfy+a\bfz]$ for any $a\in\bbF_{q^2}$.
Thus, $(\calV,\calB)$ satisfies (i), (iii) and (iv).
For (v), note that if $[\bfx]\notin[\bfy,\bfz]$ then there is a unique line in $\Span\set{\bfy,\bfz}$ that is orthogonal to $\bfx$:
since $\bfx\notin\set{\bfy,\bfz}^\perp$,
we have $0=\bfx\cdot(a\bfy+b\bfz)=a(\bfx\cdot\bfy)+b(\bfx\cdot\bfz)$ if and only if $(a,b)$ is a multiple of $(-\bfx\cdot\bfz,\bfx\cdot\bfy)$.

This leaves (ii).
Take any $[\bfx]\in\calV$.
Since the support of any nonzero self-orthogonal vector is at least two,
we may permute the indices of our vectors so as to assume without loss of generality that $\bfx=(1,a,b,c)$ where $a\neq 0$ and $a^{q+1}+b^{q+1}+c^{q+1}=-1$.
Every block contains a unique member of the ovoid $\calO=\set{[\bfy]\in\calV: y_1=0}$.
As such, the number of blocks $[\bfy,\bfz]$ that contain $[\bfx]$ equals the cardinality of $\set{[\bfy]\in\calO: \bfx\cdot\bfy=0}$.
We want to show this number is $q+1$.
Equivalently, since $[\bfy]$ consists of all nonzero scalar multiples of $\bfy$,
we want to show that the cardinality of
\begin{equation*}
\set{(0,d,e,f)\in\bbF_{q^2}^4: (0,d,e,f)\neq(0,0,0,0),\ a^qd+b^qe+c^qf=0,\ d^{q+1}+e^{q+1}+f^{q+1}=0}
\end{equation*}
is $(q+1)(q^2-1)$.
To do so, note that since $a\neq0$, $d=-[(\tfrac ba)^q e+(\tfrac ca)^q f]$ is uniquely determined by $e$ and $f$.
As such, if $(e,f)=(0,0)$ then $(0,d,e,f)=(0,0,0,0)$.
Thus, we wish to compute the number of \smash{$(e,f)\in\bbF_{q^2}^2$}, $(e,f)\neq(0,0)$ such that
$0=[(\tfrac ba)^q e+(\tfrac ca)^q f]^{q+1}+e^{q+1}+f^{q+1}$.
By multiplying by $a^{q+1}$, distributing and then using the fact that $a^{q+1}+b^{q+1}+c^{q+1}=-1$,
we equivalently are counting the number of $(e,f)\neq(0,0)$ such that
\begin{align}
\nonumber
0
&=(b^q e+c^q f)^{q+1}+a^{q+1}e^{q+1}+b^{q+1}f^{q+1}\\
\nonumber
&=(a^{q+1}+b^{q+1})e^{q+1}+(a^{q+1}+c^{q+1})f^{q+1}+b^qcef^q+(b^qcef^q)^q\\
\label{equation.appendix proof 1}
&=-(1+c^{q+1})e^{q+1}-(1+b^{q+1})f^{q+1}+b^qcef^q+bc^qe^qf.
\end{align}

In the case where $1+b^{-1}\neq0$ note this equation implies that if $e=0$ then $f=0$.
As such, in this case, every nonzero solution $(e,f)$ to~\eqref{equation.appendix proof 1} has $e\neq0$, meaning we can divide by $-(1+b^{q+1})e^{q+1}$ to obtain the equivalent equation
\begin{equation*}
0=x^{q+1}-\frac{b^qc}{1+b^{q+1}}x^q-\frac{bc^q}{1+b^{q+1}}x+\frac{1+c^{q+1}}{1+b^{q+1}}
\end{equation*}
where $x:=\frac fe$.
Completing the ``square" and noting that $y^{q+1}=y^2$ for all $y\in\bbF_q$ then gives
\begin{equation*}
0
=\biggl(x-\frac{b^qc}{1+b^{q+1}}\biggr)^{q+1}-\frac{b^{q+1}c^{q+1}}{(1+b^{q+1})^2}+\frac{1+c^{q+1}}{1+b^{q+1}}
=\biggl(x-\frac{b^qc}{1+b^{q+1}}\biggr)^{q+1}-\frac{a^{q+1}}{(1+b^{q+1})^2}.
\end{equation*}
There are exactly $q+1$ such $x$, each obtained by adding \smash{$\frac{b^qc}{1+b^{q+1}}$} to the $q+1$ distinct $(q+1)$th roots of \smash{$\frac{a^{q+1}}{(1+b^{q+1})^2}$}.
For each such $x$, there are $q^2-1$ corresponding nonzero solutions $(e,f)=(e,ex)$ to \eqref{equation.appendix proof 1}, one for each \smash{$e\in\bbF_{q^2}^\times$}, giving exactly $(q+1)(q^2-1)$ solutions total, as claimed.

By the symmetry of~\eqref{equation.appendix proof 1},
a similar argument holds in the case where $1+c^{q+1}\neq0$.
As such, all that remains is to consider the case where $1+b^{q+1}=1+c^{q+1}=0$.
Here, \eqref{equation.appendix proof 1} becomes $y+y^{q}=0$ where $y=b^qcef^q$.
In general, the equation $y+y^q=0$ has $q$ distinct solutions, namely $y=0$ and the $(q-1)$th roots of $-1$, which can be explicitly obtained by writing both $y$ and $-1$ as exponents of the generator of $\bbF_{q^2}^\times$.
When $y=0$, we either have $e=0$ where \smash{$f\in\bbF_{q^2}^\times$} is arbitrary, or $f=0$ where \smash{$e\in\bbF_{q^2}^\times$} is arbitrary.
For any of the remaining $q-1$ values of $y$, we have \smash{$f\in\bbF_{q^2}^\times$} is arbitrary and \smash{$e=\frac1{b^q cf^q}y$}.
Altogether, we again have $(q+1)(q^2-1)$ nonzero solutions $(e,f)$ to \eqref{equation.appendix proof 1}.

\section*{Acknowledgments}

We thank the two anonymous reviewers for the insightful comments and helpful suggestions.
This work was partially supported by NSF DMS 1321779, AFOSR F4FGA05076J002 and an AFOSR Young Investigator Research Program award.
The views expressed in this article are those of the authors and do not reflect the official policy or position of the United States Air Force, Department of Defense, or the U.S.~Government.

\end{document}